\documentclass[12pt]{amsart}
\usepackage{amssymb,latexsym,amsmath,enumerate,nicefrac}
\usepackage{amsaddr}

\usepackage{color}
\definecolor{myurlcolor}{rgb}{0.6,0,0}
\definecolor{mycitecolor}{rgb}{0,0,0.8}
\definecolor{myrefcolor}{rgb}{0,0,0.8}
\usepackage[bookmarks=false,pagebackref=true]{hyperref}
\hypersetup{colorlinks,
linkcolor=myrefcolor,
citecolor=mycitecolor,
urlcolor=myurlcolor}

\usepackage[all]{xy}
\usepackage{tikz}
\usetikzlibrary{calc}
\usepackage{graphics}
\usepackage{dsfont}
\usepackage{amsfonts}
\usepackage{mathtools}

\allowdisplaybreaks

\newtheorem{theorem}{Theorem}[section]
\newtheorem{lemma}[theorem]{Lemma}

\newtheorem{proposition}[theorem]{Proposition}

\theoremstyle{definition}
\newtheorem{definition}[theorem]{Definition}

\theoremstyle{remark}

\makeatletter
\newcommand{\vast}{\bBigg@{3.35}}
\newcommand{\Vast}{\bBigg@{5}}
\makeatother

\makeatletter
\newcommand*\bigcdot{\mathpalette\bigcdot@{.6}}
\newcommand*\bigcdot@[2]{\mathbin{\vcenter{\hbox{\scalebox{#2}{$\m@th#1\bullet$}}}}}
\makeatother

\newcommand{\FA}{\mathcal F_{\mathds A}}
\newcommand{\DA}{\Delta_{\mathds A}}
\newcommand{\Z}{{\mathds Z}}
\newcommand{\A}{{\mathds A}}
\newcommand{\ZA}{{\Z_\A}}
\newcommand{\e}{\mathrm e}
\renewcommand{\i}{i}

\newcommand{\M}{{\mathcal M}}

\renewcommand{\)}{\right)}

\begin{document}
\title{On Infinitesimal Generators and Feynman-Kac Integrals of Adelic Diffusion}
\author{David Weisbart}\address{Department of Mathematics\\University of California, Riverside}\email{weisbart@math.ucr.edu} 

\maketitle



\begin{abstract} 
For each prime $p$, a Vladimirov operator with a positive exponent specifies a $p$-adic diffusion equation and a measure on the Skorokhod space of $p$-adic paths.  The product, $P$, of these measures with fixed exponent is a probability measure on the product of the $p$-adic path spaces.  The adelic paths have full measure if and only if the sum, $\sigma$, of the diffusion constants is finite.  Finiteness of $\sigma$ implies that there is an adelic Vladimirov operator, $\Delta_{\mathds A}$, and an associated diffusion equation whose fundamental solution gives rise to the measure induced by $P$ on an adelic Skorokhod space. For a wide class of potentials, the dynamical semigroups associated to adelic Schr\"{o}dinger operators with free part $\Delta_{\mathds A}$ have path integral representations.
\end{abstract}


\tableofcontents

\section{Introduction}\label{back}

Weyl's framework for quantum kinematical systems permits a natural generalization of such systems with real configuration spaces to the setting where the configuration space is a locally compact, Hausdorff abelian group \cite[Ch. III. $\S$ 16, Ch. IV. $\S$ 14, $\S$ 15]{Weyl}. Quantum systems with a $p$-adic or an adelic configuration space fit into Weyl's general framework.  Volovich proposed in \cite{vol} that the geometry of spacetime might be non-Archimedean at ultra small distance and time scales because of the failure of measurement at the Planck scale and the fact that the Archimedean axiom is fundamentally a statement about subdivision of physically measurable quantities.  In their seminal articles \cite{VV89a} and \cite{VV89b}, Vladimirov and Volovich initiated the study of quantum systems in the $p$-adic setting. In \cite{zel}, Zelenov studied Feynman integrals with $p$-adic valued paths and both Parisi in \cite{par} and Meurice in \cite{Meu} proposed certain functional integrals where time is $p$-adic. The last three decades have seen considerable interest in the study of non-Archimedean physics and Dragovich, Khrennikov, Kozyrev, and Volovich provide in \cite{dra:survey} an impressive survey of research in this area with an extensive list of references.

Vladimirov introduced in \cite{Vlad88} a pseudo-differential operator analogous to the classical Laplacian and acting on certain complex valued functions with domain in the $p$-adic numbers.  He investigated the spectral properties of this operator in \cite{Vlad90}.  Prior to this, both Taibleson in \cite{Taib} and Saloff-Coste in \cite{SC1} wrote about pseudo-differential operators in the context of local fields.  Saloff-Coste studied such operators in \cite{SC2} in the more general setting of local groups.  Kochubei gave in \cite{koch92} the fundamental solution to the $p$-adic analog of the diffusion equation, with the Vladimirov operator, discussed by Vladimirov in \cite{Vlad88, Vlad90}, replacing the Laplace operator. He furthermore developed a theory of $p$-adic diffusion equations and proved a Feynman-Kac formula for the operator semigroup with a $p$-adic Schr\"{o}dinger operator as its infinitesimal generator. The contemporary burgeoning study of $p$-adic quantum mechanics at least in part motivated the study of diffusion in the $p$-adic setting in \cite{alb}.  In this work, Albeverio and Karwowski constructed a continuous time random walk on $\mathds Q_p$, computed its transition semigroup and infinitesimal generator, and showed among other things that the associated Dirichlet form is of jump type.   Already in \cite{blair}, Blair studied diffusion in the context of the ring of rational adeles.  Torba and Z\'{u}\~{n}iga-Galindo further explored adelic diffusion in a recent article \cite{Zun:a}, introducing a metric structure under which the adeles become a complete metric space.  They used this metric structure to define pseudo-differential operators and an associated family of parabolic type equations whose fundamental solutions give rise to the transition functions of a Markov semigroup.  Seeking to better understand the adelic Feynman integrals discussed by Djordjevi\'{c}, Dragovich, and Ne\v{s}i\'{c} in \cite{DDN} and Dragovich and Raki\'{c} in \cite{DR}, we investigate the adelic path integral in the probabilistic setting.  We will follow the approach of \cite{var97} for the construction of adelic diffusion.  While our operator appears to be different from the one Torba and Z\'{u}\~{n}iga-Galindo present in \cite{Zun:a}, we answer at least in the present context their call to study adelic Schr\"{o}dinger equations together with associated Feynman and Feynman-Kac integrals.

The path integral formulation of quantum mechanics due to Feynman involves the concept of a complex valued measure on an infinite dimensional path space \cite{feyn, lap}.  However, Cameron showed in \cite{cam} that there are no such measures.  The probability measures on path spaces that appear in the setting of diffusion are deeply linked to the desired measures in the quantum framework and are potentially useful for making mathematically precise some of the heuristics of the quantum setting.  The mathematical theory of diffusion plays a central role in probability theory and is an important subject of study for its intrinsic interest and for its many applications.  While the significance of adelic diffusion to the foundation of quantum theory is a primary motivator of the present study, the results of the present study should have applications to more general analytical problems in the adelic setting.  Torba and Z\'{u}\~{n}iga-Galindo point out in \cite{Zun:a} some examples of research that motivate further study of pseudo-differential equations on the ring of rational adeles, namely, Haran's discovery in \cite{Haran} of a connection between explicit formulas for the Riemann zeta function and adelic pseudo-differential operators and Connes' study of the Riemann zeta function in \cite{Connes}.

Denote henceforth by $\mathds Q_p$ the $p$-adic numbers.  Section~\ref{two} reviews basic properties of $\mathds Q_{p}$, the ring of adeles, and measures on path spaces.  While we refer to \cite{DWexit} for most of this background, we discuss some additional required background and establish further notational conventions.  Suppose that $(p_i)$ is the strictly increasing enumeration of the prime numbers.  Suppose that $I$ is a time interval, a closed but possibly unbounded subinterval of $\mathds R$ whose minimum element is 0.  Suppose that $D(I\colon \mathds Q_{p_i})$ is the Skorokhod space of paths valued in $\mathds Q_{p_i}$.  Suppose that $(D(I\colon \mathds Q_{p_i}), P^i)$ is the probability space discussed in \cite{DWexit} for a diffusion process that is valued in $\mathds Q_{p_i}$, that gives full measure to paths starting at 0, and that corresponds to a Vladimirov operator with exponent $b$ and diffusion constant $\sigma_i$.  Parameters $I$ and $b$ are not permitted to vary in the index $i$.  For each $x_i$ in $\mathds Q_{p_i}$, the measure $P^i_{x_i}$ gives full measure to paths starting at $x_i$ and agrees with $P^i$ up to a translation by $x_i$.

Suppose that $x$ is in the ring of adeles and that $x_i$ is the $i^{\rm th}$ component of $x$.  Denote by $D(I)$ the infinite product of $D(I\colon \mathds Q_{p_i})$ and by $P_x$ the product measure $\otimes_{i}P^{\i}_{x_i}$. Suppose that $X^{\i}$ is the stochastic process that maps each $t$ in $I$ to the random variable $X^{\i}_t$ that acts on paths $\omega_i$ in $(D(I\colon \mathds Q_{p_i}), P^i)$ by \[X^{\i}_t(\omega_i) = \omega_i(t).\] Define by $Y$ the stochastic process that maps each $t$ in $I$ to a random variable acting on a path $\omega$ in the probability space $(D(I), P)$ by \[Y_t(\omega) = (X^{\i}_t(\omega_i)) = (\omega_i(t)) \quad {\rm where}\quad \omega = (\omega_i).\] The exit probabilities of \cite{DWexit} lead to the main result of Section~\ref{three}, that the stochastic process $Y$ has a version in the Skorokhod space of paths valued in the adeles if and only if the sequence $(\sigma_i)$ is summable.   Denote by $N_T$ the random variable that counts the number of $p_i$-adic components in which a sample path of $Y$ that originates at 0 wanders outside of the ring of $p_i$-adic integers by time $T$, where $T$ is in $I$.  Section~\ref{three} additionally proves that $N_T$ has finite moments of all orders.  Section~\ref{four} proves that the summability condition also implies that the adelic versions of the processes come from adelic diffusion equations.  Furthermore, the measure on the adelic Brownian bridges agrees with the product measure whose factors are measures on Brownian bridges given by an adelic conditioning. Section \ref{five} presents a Feynman-Kac formula for the dynamical semigroup associated to an adelic Schr\"{o}dinger operator.  If the potential is a bounded simple adelic potential, then the adelic Feynman-Kac integral is a product of integrals over each component.  In this case, the dynamical semigroup is a semigroup of integral operators whose kernel is a product of kernels in each component, as described in the Feynman setting in \cite{DDN}.

The current work should find immediate application in several areas.  Schwinger studied the foundations of quantum mechanics in a series of papers \cite{Sch1} and gave a systematic account of these foundational ideas in \cite{Sch2}.  Among these papers, Schwinger studied the finite dimensional approximation of quantum mechanical systems by working with finite configuration spaces in \cite{Sch3}.  Digernes, Hustad, and Varadarajan explored more generally the finite approximation of Weyl systems in \cite{DHV}.  Digernes, Varadarajan, and Varadhan in \cite{DVV} extended these kinematical approximations in the setting of a real configuration space and its grid approximations to the dynamical setting that Schwinger considered.  Varadarajan discussed in \cite{VSV-SW} some generalizations of the quantum systems of Weyl and Schwinger in the setting of locally compact abelian groups.  Albeverio, Gordon, and Khrennikov studied finite approximation of quantum systems in locally compact abelian groups in \cite{agk}, obtaining some of the results of \cite{DVV} in this very general setting.  In \cite{BDW}, Bakken, Digernes, and Weisbart improved upon some of the results of \cite{agk}, but in the restricted setting of configuration spaces that are local fields.  The work on finite approximation of quantum systems is ongoing with many remaining open questions.  The results of the current paper should make it possible to extend many previous results about finite approximation in the real and $p$-adic settings to the adelic setting.  Ultrametricity arises in the theory of complex systems and many references cited by \cite[Chapter~4]{KKZ} may be extended to an adelic setting, for example, the works \cite{ave2, ave3, ave4, ave} of Avetisov, Bikulov, Kozyrev, and Osipov dealing with $p$-adic models for complex systems.


\section{Background}\label{two}

We follow closely the notation and conventions of \cite{DWexit}, repeating here only that which seems necessary because of some minor changes in notation required by the fact that we must deal with all primes rather than a single fixed prime.  Refer to \cite{DWexit} for a discussion of path spaces and probability measures on path space.

\subsection{Basic Facts about $\mathds Q_p$ and $\mathds A$}\label{basic1}

For each natural number $i$, denote by $\mathds Q_{p_i}$ the field of $p_i$-adic numbers with the absolute value denoted by $|\cdot|_i$.  Denote respectively by $B^i_k(x)$ and $S^i_k(x)$ the ball and the circle of radii $p_i^k$, the compact open sets \[{B}^i_k(x) = \{y\in \mathds Q_{p_i}\colon |y-x|_i \leq p_i^k\} \quad {\rm and}\quad {S}^i_k(x) = \{y\in \mathds Q_{p_i}\colon |y-x|_i = p_{i}^k\}.\]  Denote by $\mathds Z_{p_i}$ the \emph{ring of integers}, the unit ball in $\mathds Q_{p_i}$, and by $\mu_i$ the Haar measure on $\mathds Q_{p_i}$ that is normalized to be one on $\mathds Z_{p_i}$.  Denote by $\mathds S^1$ the unit complex numbers and by $\chi_i$ the rank 0 character \[\chi_i\colon \mathds Q_{p_i} \to \mathds S^1\] that is simply denoted by $\chi$ in \cite{DWexit} for a fixed prime $p$.  Denote by $\mathcal F_i$ the Fourier transform $L^2(\mathds Q_{p_i})$ and by $\mathcal F_i^{-1}$ its inverse. Denote by $SB(\mathds Q_{p_i})$ the \emph{Schwartz-Bruhat} space of complex valued, compactly supported, locally constant functions on $\mathds Q_{p_i}$.

\begin{definition}
Elements of an infinite set $I$ are \emph{almost always} in a set $J$ if all but finitely elements of $I$ are in $J$.  A property holds for \emph{almost all} $i$ in an infinite set $I$ if the property holds for all but finitely many elements of $I$.
\end{definition}

The ring of  \emph{rational adeles}, denoted by $\mathds A$, is the set %
\begin{align*}\mathds A = \left\{(a_1, a_2, \dots)\colon \forall i\in \mathds N,\; a_i \in \mathds Q_{p_i} \text{ and for almost all } i\in \mathds N,\; a_i\in \mathds Z_{p_i}\right\}.\end{align*}  %
This is to say that the ring $\A$ is the restricted product of the $\mathds Q_{p_i}$ with respect to the subrings $\mathds Z_{p_i}$.  The ring of  \emph{completed rational adeles} has an additional component, a first component $a_\infty$ which is in the real numbers.  While the main results of this paper will hold over the ring of completed rational adeles as can be verified by the reader, the current work deals exclusively with the ring of rational adeles and refers to them henceforth as the \emph{adeles}.  The topology on the adeles will be the restricted product topology, namely, the topology generated by sets that are products of balls in each $p_i$-adic component and that are component-wise almost always equal to $\mathds Z_{p_i}$. The following proposition and proof are well known, but included for the readers' convenience.

\begin{proposition}\label{polish}
A Hausdorff topological space that is second countable and locally compact is completely metrizable.
\end{proposition}

\begin{proof}
Suppose that $X$ is locally compact, Hausdorff, and second countable and denote by $Y$ the one point compactification of $X$.  The topological space $Y$ is compact and Hausdorff, hence regular.  The second countability of $X$ implies that $Y$ is second countable and so the Urysohn metrization theorem implies that $Y$ is metrizable.  The compactness of $Y$ implies that it is complete in any metric that gives rise to its topology and so $X$ is completely metrizable as an open subset of $Y$.
\end{proof}

Since $\A$ with the restricted product topology is a second countable, locally compact, Haudorff topological space, Proposition~\ref{polish} implies that $\A$ is a Polish space.  Note that \cite{Zun:a} constructs a complete metric on the ring of completed rational adeles that will induce a complete metric on $\A$, however, the current study requires no such specification of a metric.  While in their discussion the metric plays a critical role in the construction of an adelic pseudo-Laplace operator, our adelic pseudo-Laplace operator will be different.

Define by $\Z_{\A}$ the set of all elements $\left(a_i\right)$ of $\A$ such that $a_i$ is in $\Z_{p_i}$ for every natural number $i$.  Under the usual component-wise addition and multiplication, the ring of adeles forms a ring and $\Z_{\A}$ is an open, compact subring.  As a locally compact abelian group under its addition operation, the ring of adeles has a Haar measure $\mu_{\A}$ that is unique on fixing a normalization.  Normalize $\mu_{\A}(\ZA)$ to be equal to 1.  The measure $\mu_\A$ on $\A$ is the infinite product of the measures $\mu_i$ taken over the index $i$.

For each natural number $i$, denote by $\mathcal H_i$ the Hilbert space $L^2\big(\mathds Q_{p_i}, \mu_i\big)$ and specify the vacuum vector $\Omega_i$ of $\mathcal H_i$ to be the characteristic function on $\mathds Z_{p_i}$.  Define by $\mathcal H_{\rm alg}$ the set \[\mathcal H_{\rm alg} = \underset{i\in\mathds N}{\overset{\bigcdot}{\bigotimes}}\, \mathcal H_{i},\] the infinite algebraic tensor product whose $i^{\rm th}$ component is an element of $\mathcal H_{p_i}$ and whose components are almost always vacuum vectors.  An element $f$ of $\mathcal H_{\rm alg}$ is \emph{simple} if for every natural number $i$ there is a function $f_i$ in $\mathcal H_{i}$ with the property that \[f = f_1\otimes f_2 \otimes \cdots\] and $f_i$ is equal to a vacuum vector for almost all $i$.  Arbitrary elements of $\mathcal H_{\rm alg}$ are finite linear combinations of simple functions.  For each simple $f$ and $g$ in $\mathcal H_{\rm alg}$, define the inner product $\langle f, g\rangle_{\mathds A}$ and the norm $||f||_{\mathds A}$ by \[\langle f, g\rangle_{\mathds A} = \prod_{i\in\mathds N}\langle f_i, g_i\rangle_{i} \quad {\rm and} \quad ||f||_{\mathds A} = \prod_{i\in\mathds N}||f_i||_{i},\] extending these to arbitrary elements of $\mathcal H_{\rm alg}$ by linearity.  The product $\langle\cdot, \cdot\rangle_{\mathds A}$ is an inner product on $\mathcal H_{\rm alg}$ and $||\cdot||_{\mathds A}$ is its associated norm.  The restricted tensor product, $\mathcal H$, of the $\mathcal H_{i}$ with respect to the specified vacuum vectors is the analytic completion under $||\cdot||_{\mathds A}$ of the above infinite algebraic tensor product. It is in this sense that \[\mathcal H = \bigotimes_{i\in\mathds N}\mathcal H_{i}.\]

Identify the Hilbert space $L^2(\A, \mu_\A)$ of square integrable functions on $\A$ with $\mathcal H$ and view any operator on $L^2(\A, \mu_\A)$ equivalently as an operator on $\mathcal H$.  An element $f$ of $\mathcal H$ is a \emph{simple adelic Schwartz-Bruhat function} if it is a simple element of $\mathcal H_{\rm alg}$ and the component functions of its product representation are Schwartz-Bruhat functions. The adelic Schwartz-Bruhat space, $SB({\mathds A})$, is the set of finite sums of simple adelic Schwartz-Bruhat functions and is a dense subset of $\mathcal H$.  A locally compact, Hausdorff, abelian group, the set of rational adeles is equipped with a Fourier transform.  For each simple function $f$ in $SB({\mathds A})$ and for each $y$ in $\A$ with $y$ equal to $(y_i)$, define the Fourier transform $\FA$ on $f$ by \[(\mathcal F_{\A}f)(y) = \prod_{i\in\mathds N}\big(\mathcal F_if_i\big)(y_i)\] and extend it to act on all of $SB({\mathds A})$ by linearity.   As a unitary operator on the densely defined subspace $SB({\mathds A})$ of $\mathcal H$, the operator $\FA$ extends to a unitary operator on all of $\mathcal H$.  Define similarly the inverse $\FA^{-1}$ of $\FA$.

\subsection{Diffusion in the $\mathds Q_p$ Setting}\label{basic2}

Fix a positive real number $b$.  For each natural number $i$, the multiplication operator ${\mathcal M}_i$ acts on $SB(\mathds Q_{p_i})$ by \[({\mathcal M}_if)(x) = |x|_i^bf(x).\]  Denote by $\Delta_i$ the unique self adjoint extension of the pseudo Laplace operator $\tilde{\Delta}_i$ with exponent $b$ that acts on $SB(\mathds Q_{p_i})$ by \begin{equation*}\big(\tilde{\Delta}_i f\big)(x) = \big(\mathcal F_i^{-1}{\mathcal M}_i\mathcal F_if\big)\!(x).\end{equation*} Extend this operator to the \emph{Vladimirov operator with exponent} $b$ that is denoted $\hat{\Delta}_i$ and that acts on functions $f$ on $\mathds R_+\times \mathds Q_{p_i}$ that for each positive $t$ are in the domain of $\Delta_i$ by \begin{equation*}(\hat{\Delta}_if)(t,x) = (\Delta_if(t,\cdot))(x).\end{equation*}  Denote ambiguously by $\Delta_i$ the operator $\hat{\Delta}_i$. For each natural number $i$, fix $\sigma_i$ to be a positive real number and refer to it as a diffusion constant.  The pseudo differential equation \begin{equation}\label{eq1} \dfrac{{\rm d}f(t,x)}{{\rm d}t} = -\sigma_i\Delta_i f(t,x)\end{equation} has as its fundamental solution the function \begin{equation}\rho^i(t,x) = \left(\mathcal F_i^{-1}e^{-\sigma_it|\cdot|_i^b}\right)\!(x).\end{equation}  %

A minor modification of the more general arguments of \cite{var97} show that if $I$ is any time interval, then for each positive $t$ in $I$ the function $\rho^i(t,\cdot)$ is a probability density function and that $\rho^i$ gives rise to a probability measure $P^i$ on $D(I\colon \mathds Q_{p_i})$ that is concentrated on the set of paths originating at 0.  A history, $h$, for paths in $D(I\colon \mathds Q_{p_i})$ is a finite sequence of the form \[h = ((0,U_0), (t_1, U_1), \dots, (t_n, U_n)),\] where $n$ is a natural number, $(t_i)$ is a strictly increasing finite sequence in $I$, and $(U_i)$ is a finite sequence of Borel subsets of $\mathds Q_{p_i}$.  The sequence $(t_k)_{k=1}^n$ is an \emph{epoch} and the sequence $(U_k)_{k=0}^n$ is a \emph{route}.  The natural number $n$ is the \emph{length} of the epoch.  Denote by $C(h)$ the subset of $D(I\colon \mathds Q_{p_i})$ given by \[C(h) = \big\{\omega \in D(I\colon \mathds Q_{p_i})\colon \omega(0) \in U_0\; {\rm and}\; \forall j \in \{1, \dots, n\},\;\omega(t_j)\in U_j\big\}.\] Sets of the form $C(h)$ for some history $h$ are \emph{simple cylinder sets}.  The \emph{set of cylinder sets} is the $\sigma$-algebra generated by the simple cylinder sets.  If $U_0$ is the set $\{0\}$, then define $P^i(C(h))$ by \begin{align}\label{meas}P^i(C(h)) &= \int_{U_1} \cdots \int_{U_n} \rho^i(t_1, x_1) \rho^i(t_2 - t_1, x_2 - x_1)\\&\qquad\qquad\qquad\qquad\cdots \rho^i(t_n - t_{n-1}, x_n - x_{n-1}) \,{\rm d}\mu_i\!\(x_n\)\cdots \,{\rm d}\mu_i\!\(x_1\).\notag\end{align}  If 0 is not in $U_0$, then $P^i(C(h))$ is 0.  This premeasure on the $\pi$-system of simple cylinder sets extends to a probability measure on the cylinder sets of $D(I \colon \mathds Q_{p_i})$ that we will once again and henceforth denote by $P^i$.  For further discussion, see \cite{bw}. Define the stochastic process $X^i$ on the probability space $(D(I \colon \mathds Q_{p_i}), P^i)$ to be the function \begin{align}\label{pProc}X^i \colon I\times D(I\colon \mathds Q_{p_i}) \to \mathds Q_{p_i} \quad {\rm by}\quad (t, \omega) \mapsto X^i_t(\omega) = \omega(t).\end{align}

For each time interval $I$, the probability measure $P^i$ on $D(I \colon \mathds Q_{p_i})$ gives full measure to paths originating at 0.   For each $x_i$ in $\mathds Q_{p_i}$, define $P^i_{x_i}$ to be the probability measure given by the same density function that defines $P^i$ but conditioned to give full measure to the paths originating at $x_i$.  If $A$ is a simple cylinder set, then \[P^i_{x_i}(A) = P^i(A - x_i).\] For each $y_i$ in $\mathds Q_{p_i}$ and each positive $T$ in $I$, the arguments of \cite{var97} guarantee the existence of the probability measures concentrated on the $p_i$-adic Brownian bridges, namely, the measures $P^i_{T, x_i, y_i}$ which are given by the measures $P^i_{x_i}$ conditioned so that paths almost surely take value $y_i$ at time $T$.  These conditioned measures form a continuous family of probability measures depending on the starting and ending points, as discussed in \cite{var97} in a more general setting but with the diffusion constant $\sigma_i$ restricted to be equal to 1.  There is no obstruction to allowing for a more general diffusion constant, which \cite{bw} discusses.


\section{Products of $p$-Adic Path Spaces}\label{three}

\subsection{Exit Time Probabilities}

Suppress the index $i$ in fixing a prime $p_i$ in this subsection.  Following \cite{DWexit}, denote by $||X||_T$ the value \[||X||_T=\sup_{0\leq t\leq T}|X_t|\] and by $\alpha$ the quantity \begin{align}\label{alpha}\alpha = 1 - \dfrac{p^b-1}{p^{b+1}-1}.\end{align} Let $r$ be an integer.  Theorem~3.1 of \cite{DWexit} states that for any non-negative real number $T$, \[P\left(||X||_T\leq p^r\right) = {\rm e}^{-\sigma\alpha Tp^{-rb}}.\] We specialize this result in the following proposition.

\begin{proposition}\label{three:ExitEqual}
For any non-negative real number $T$, \[P\left(||X||_T\leq 1\right) = {\rm e}^{-\sigma\alpha T}.\] 
\end{proposition}

Suppose that $x$ is in $\mathds Q_p$ and $y$ is in $x+p^{-r}\mathds Z_p$.  Theorem~4.7 of \cite{DWexit} states that for all $t$ in $(0,T]$, \[P_{t, x, y}\big(||X-x||_T\leq p^r\big) \ge P_x(||X-x||_T\leq p^r).\] We specialize this result in the following proposition.

\begin{proposition}\label{three:ExitCondInequal}
Suppose that $x$ is in $\mathds Q_p$ and $y$ is in $x+\mathds Z_p$.  For all $t$ in $(0,T]$, \[P_{t, x, y}\big(||X-x||_T\leq 1\big) \ge P_x(||X-x||_T\leq 1).\] 
\end{proposition}

\subsection{The Probability of the Adelic Path Subspace}
Fix henceforth a time interval $I$, an exponent $b$, and a sequence $(\sigma_i)$ of diffusion constants.  The time interval and exponent are constant in the index $i$.  Denote by $D(I)$ the product space \[D(I) = \prod_{i\in\mathds N} D(I\colon \mathds Q_{p_i}).\] Denote by $D(I\colon \mathds A)$ the set of Skorokhod paths valued in the adeles. The natural inclusion map permits an identification of $D(I\colon \mathds A)$ with the subset of $D(I)$ whose components are almost always in $\mathds Z_{p_i}$ for the respective index $i$.  Define by $P^{i}$ the measure on $D(I\colon \mathds Q_{p_i})$ given by \eqref{meas} and by $P$ the product measure $\otimes_{i}P^i$ on $D(I)$.  Denote by $\sigma$ the possibly infinite quantity \[\sigma = \sum_{i\in \mathds N} \sigma_i.\]

\begin{theorem}\label{DAfullMeas}
If $\sigma$ is finite, then $D(I\colon \mathds A)$ has full measure in $D(I)$. 
\end{theorem}

\begin{proof}
For any positive real number $T$ in $I$, denote by $D_M^T$ and $\sigma_M$ respectively the set and the finite sum given by \[D_M^T = \Big\{\omega\in D(I)\colon \sup_{0\leq t\leq T}|\omega_p(t))|\leq 1, \;\forall i \ge M\Big\} \quad {\rm and} \quad \sigma_M = \sum_{i\ge M}\sigma_i.\] Since $\alpha_i$ is in $(0,1)$ for each $i$, Proposition~\ref{three:ExitEqual} together with the algebraic properties of the exponential imply that \[P\big(D_M^T\big) \ge e^{-\sigma_MT}.\]  Under the assumption that $\sigma$ is finite, $\sigma_M$ tends to 0 as $M$ tends to infinity and so \[\lim_{M\to \infty} P\big(D_M^T\big) = 1.\] Denote by $D^T$ the set given by \[D^T = \bigcup_{M\in\mathds N} D_M^T.\]  Continuity from below of the measure $P$ implies that $P\big(D^T\big)$ is equal to 1.  If $I$ is a bounded interval $[0, S]$, then take $T$ to equal $S$ so that $D(I\colon \mathds A)$ is equal to $D^T$ and so $P(D(I\colon \mathds A))$ is equal to 1. Otherwise, $I$ is the interval $[0,\infty)$ and so \[D(I\colon \mathds A) = \bigcap_{T\in \mathds N}D^T,\] hence $P(D(I\colon \mathds A))$ is equal to 1.
\end{proof}

\subsection{Moments of the $\ZA$ First Exit Number}

Denote by $Y$ the stochastic process that maps each positive $t$ in $I$ to the random variable $Y_t$ acting on the probability space $\(D(I),P\)$ by \[Y_t(\omega) = \omega(t).\] For each positive $T$ in $I$, define the random variable $N_T^{\i}$ on the probability space $\(D(I),P\)$ in the following way.  If $\omega$ is in $D(I)$, then $N_T^{\i}(\omega)$ is 0 if the $i^{\rm th}$ component of $\omega$ remains in $\mathds Z_{p_i}$ for each $t$ less than or equal to $T$.  Otherwise, $N_T^{\i}(\omega)$ is 1.  Denote by $N_T$ the sum of the $N_T^{\i}$ over all $i$, where $N_T$ can potentially take on the value infinity. The random variable $N_T$ counts the number of $p_i$-adic components in which a sample path of $Y$ wanders outside of the ring of $p_i$-adic integers by time $T$ and so Theorem \ref{DAfullMeas} implies that $N_T$ is almost surely finite.

\begin{proposition}
If $\sigma$ is finite, then all moments of $N_T$ are finite.
\end{proposition}

\begin{proof}
Denote by $A$ the subset of all paths in $D(I)$ with the property that $\omega$ is in $A$ if and only if $\omega$ remains in $\ZA$ for all $t$ in $[0,T]$.  Denote by $\alpha_i$ and $\beta_i$ the quantities \[\alpha_i = 1 - \dfrac{p_i^b-1}{p_i^{1+b}-1} \quad {\rm and}\quad \beta_i = \sigma_i\alpha_i.\]  Let $\mathcal S_k$ be the set of all sequences of length $k$ that are valued in $\mathds N$ and that have distinct values in each of their $k$ places.  Denote by $\beta$ the sum \[\beta = \sum_i \beta_i.\] Independence of the components of the stochastic process $Y$ implies that %
\begin{align*}
P(N_t = k) &= \sum_{\substack{(i_1, \dots, i_k)\in \mathcal S_k\\i_1 < \cdots <i_k}}\big(1- \e^{-T\beta_{i_1}}\big) \cdots \big(1 - \e^{-T\beta_{i_k}}\big)P(A) \e^{T\beta_{i_1}}\cdots \e^{T\beta_{i_l}} 
\\& = \frac{1}{k!}\sum_{(i_1, \dots, i_k)\in  \mathcal S_k}\big(\e^{T\beta_{i_1}}-1\big) \cdots \big(\e^{T\beta_{i_k}}-1\big)P(A)
\\& = \frac{P(A)}{k!}\sum_{(i_1, \dots, i_k)\in  \mathcal S_k}\big(\e^{T\beta_{i_1}}-1\big) \cdots \big(\e^{T\beta_{i_{k-1}}}-1\big)\big(\e^{T\beta_{i_k}}-1\big).%
\end{align*}
The inequality \[\big(\e^{T\beta_{i_k}}-1\big) < \sum_{j\in \mathds N}\big(\e^{T\beta_{j}}-1\big)\] implies that 
\begin{align*}
P(N_t = k) & < \frac{P(A)}{k!}\sum_{(i_1, \dots, i_{k-1})\in  \mathcal S_{k-1}}\big(\e^{T\beta_{i_1}}-1\big) \cdots \big(\e^{T\beta_{i_{k-1}}}-1\big)\bigg(\sum_{j\in \mathds N}\big(\e^{T\beta_{j}}-1\big)\bigg)%
\\& = \frac{P(A)}{k!}\sum_{(i_1, \dots, i_{k-1})\in  \mathcal S_{k-1}}\bigg(\big(\e^{T\beta_{i_1}}-1\big) \cdots \big(\e^{T\beta_{i_{k-1}}}-1\big)\phantom{\frac{(T\beta)^3}{3!}}\\&\qquad\qquad\qquad\qquad\qquad\qquad\qquad\qquad \cdot \sum_{j\in \mathds N}\Big(T\beta_j + \frac{\big(T\beta_j\big)^2}{2!}+ \frac{\big(T\beta_j\big)^3}{3!} + \cdots\bigg)\bigg)%
\\& < \frac{P(A)}{k!}\sum_{(i_1, \dots, i_{k-1})\in  \mathcal S_{k-1}}\bigg(\big(\e^{T\beta_{i_1}}-1\big) \cdots \big(\e^{T\beta_{i_{k-1}}}-1\big)\phantom{\frac{(T\beta)^3}{3!}}\\&\qquad\qquad\qquad\qquad\qquad\qquad\qquad\qquad \cdot\bigg(T\beta + \frac{(T\beta)^2}{2!}+ \frac{(T\beta)^3}{3!} + \cdots\bigg)\bigg)
\\& < \frac{\e^{T\beta}}{k}\frac{P(A)}{(k-1)!}\sum_{(i_1, \dots, i_{k-1})\in S_{k-1}}\Big(\big(\e^{T\beta_{i_1}}-1\big) \cdots \big(\e^{T\beta_{i_{k-1}}}-1\big)\Big)
\\& = \frac{\e^{T\beta}}{k}P(N_t = k-1).
\end{align*}
A straightforward induction argument shows that \[P(N_T = k) < P(A)\frac{\left(\e^{T\beta}\right)^k}{k!}.\]  The estimates for $P(N_t = k)$ afford an estimate of the mean of $N_T$.  In particular, 
\begin{align*}{\rm E}\big[N_T\big] &= \sum_{k\in\mathds N} kP(N_t = k) \\ &< P(A)\sum_{k\in\mathds N} k\frac{\big(\e^{T\beta}\big)^k}{k!} \\ &= P(A)\sum_{k\in\mathds N} \frac{\big(\e^{T\beta}\big)^k}{(k-1)!} \\ &= P(A)\e^{T\beta}\sum_{k\in\mathds N} \frac{\big(\e^{T\beta}\big)^k}{k!} = P(A)\e^{T\beta}\e^{\e^{T\beta}}.
\end{align*}  Similar estimates are possible for all moments of $N_T$. If $m$ is a natural number, then%
\begin{align*}
{\rm E}\big[N_T^m\big] &= \sum_{k\in\mathds N} k^mP(N_T = k) \\ &< P(A)\sum_{k\in\mathds N} k^m\frac{\big(\e^{T\beta}\big)^k}{k!} 
\\&= P(A)\vast(\sum_{\substack{k\in\mathds N\\k< m}} k^m\frac{\big(\e^{T\beta}\big)^k}{k!} + \sum_{\substack{k\in\mathds N\\k\ge m}} k^m\frac{\big(\e^{T\beta}\big)^k}{k!}\vast)
\\& = P(A)\vast(\sum_{\substack{k\in\mathds N\\k< m}} k^m\frac{\left(\e^{T\beta}\right)^k}{k!} + \sum_{k\in \mathds N_0} (m+k)^m\frac{\left(\e^{T\beta}\right)^{(k+m)}}{(m+k)!}\vast)
\\& = P(A)\vast(\sum_{\substack{k\in\mathds N\\k< m}} k^m\frac{\left(\e^{T\beta}\right)^k}{k!} + \left(\e^{T\beta}\right)^{m}\sum_{k\in \mathds N_0} (m+k)^m\frac{\left(\e^{T\beta}\right)^{k}}{(m+k)!}\vast)
\\& = P(A)\vast(\sum_{\substack{k\in\mathds N\\k< m}} k^m\frac{\left(\e^{T\beta}\right)^k}{k!} + \left(\e^{T\beta}\right)^{m}\sum_{k\in \mathds N_0} \frac{(m+k)^m}{(m+k)(m+k-1)\cdots(k+1)}\frac{\left(\e^{T\beta}\right)^{k}}{k!}\vast)
\\& < P(A)\vast(\sum_{\substack{k\in\mathds N\\k< m}} k^m\frac{\left(\e^{T\beta}\right)^k}{k!} + \left(\e^{T\beta}\right)^{m}\max_k{\frac{(m+k)^m}{(m+k)(m+k-1)\cdots(k+1)}}\sum_{k\in \mathds N_0} \frac{\left(\e^{T\beta}\right)^{k}}{k!}\vast)
\\& = P(A)\vast(\sum_{\substack{k\in\mathds N\\k< m}} k^m\frac{\left(\e^{T\beta}\right)^k}{k!} + \left(\e^{T\beta}\right)^{m}\max_k{\frac{(m+k)^m}{(m+k)(m+k-1)\cdots(k+1)}}\e^{\e^{T\beta}}\vast)
\\& < P(A)\vast(\sum_{\substack{k\in\mathds N\\k< m}} k^m\frac{\left({\e}^{T\beta}\right)^k}{k!} + \left({\e}^{T\beta}\right)^{m}\frac{m^m}{m!}{\e}^{{\e}^{T\beta}}\vast).
\end{align*}
\end{proof}


\section{An Adelic Diffusion Equation}\label{four}

Given the summability of the diffusion constants, we show in this section that the adelic path measures of the previous section are the measures associated to adelic diffusion equations.  Recall that $\Omega_i$ is the characteristic function on $\mathds Z_{p_i}$.

\begin{lemma}\label{five:lem:cestimate}
The estimate
\[\frac{1}{\sqrt{2}} < \left|\left|\mathcal M_i\Omega_i\right|\right|_i < \sqrt{2}\] is independent of the parameter $b$.
\end{lemma}

\begin{proof}
Write an integral over $\mathds Q_{p_i}$ as a sum of integrals over circles to obtain the equalities 
\begin{align}\label{five:eq:intoverqpofMpOmega}
\big|\big|\mathcal M_i\Omega_i\big|\big|_i^2 & = \int_{\mathds Q_{p_i}}\big(|x|_i^b\Omega_i(x)\big)^2\,{\rm d}x\notag\\& = \int_{\mathds Z_{p_i}} |x|_i^{2b}\,{\rm d}x\notag\\& = \int_{S^i_0} |x|_i^{2b}\,{\rm d}x + \int_{S^i_{-1}} |x|_i^{2b}\,{\rm d}x + \int_{S^i_{-2}} |x|_i^{2b}\,{\rm d}x + \cdots.
\end{align}
Together with \eqref{five:eq:intoverqpofMpOmega}, the equalities \begin{align*}\int_{S^i_{-k}} |x|_i^{2b}\,{\rm d}x &= \big(p_i^{-k}\big)^{2b}{\rm vol}\big(S^i_{-k}\big) \\&= \big(p_i^{-k}\big)^{2b}{p_i}^{-k}\Big(1 - \frac{1}{p_i}\Big) = p_i^{-(2b+1)k}\Big(1 - \frac{1}{p_i}\Big),\end{align*} imply that 
\begin{align}\label{term:five:a}
\big|\big|\mathcal M_i\Omega_i\big|\big|_i^2 & = \sum_{k= 0}^\infty p_i^{-(2b+1)k}\Big(1 - \frac{1}{p_i}\Big)\notag\\ & = \Big(1 - \frac{1}{p_i}\Big)\frac{1}{1 - p_i^{-(2b+1)}} = \Big(1 - \frac{1}{p_i}\Big)\frac{p_i^{2b+1}}{p_i^{2b+1}-1}\end{align} The term $\big(1 - \frac{1}{p_i}\big)$ increases to 1 in $i$ and is bounded below by $\frac{1}{2}$ where the minimum is achieved when $p_i$ is 2.  For fixed $b$, the term $\frac{p_i^{2b+1}}{p_i^{2b+1}-1}$ decreases to 1 in $p_i$ and attains its maximum when $p_i$ is 2.  This second term in the product on the right hand side of \eqref{term:five:a} is also decreasing to 1 in $b$.  Since $b$ is assumed positive, this second term in the product is strictly bounded above by its value when $p_i$ is 2 and $b$ is 0, that is, it is bounded above by 2.  The inequality  \[\frac{1}{2} < \big|\big|\mathcal M_i\Omega_i\big|\big|_i^2 < 2\] therefore holds for any prime $p_i$ and any exponent $b$, thus proving the lemma.
\end{proof}

The unnecessarily rough bounds established by Lemma \ref{five:lem:cestimate} are sufficient for the goal of proving Proposition~\ref{five:prop:pretheorem} below.  Denote by $c_i$ the positive real number with the property that \[\big|\big|\M_i\Omega_i\big|\big|_i = c_i.\]  Define the multiplication operator $\M$ on each function $f$ on $\A$ by \[({\M}f)\big(a_1,\dots, a_n, \dots\big) = \big(\sigma_1|a_1|_{1}^b +\sigma_2|a_2|_{2}^b + \cdots\big)f\big(a_1,\dots, a_n, \dots\big).\] Recall that $|\cdot|_i$ is the $p_i$-adic absolute value, so $|\cdot|_1$ is the 2-adic absolute value, $|\cdot|_2$ is the 3-adic absolute value, and so on.  Denote by $\Delta_{\A}$ the operator \[\big(\Delta_\A f\big)(x)  = \big(\FA \M\FA^{-1}f\big)(x).\] Let $\Omega_\A$ be the characteristic function on $\mathds Z_\A$.  Denote an element $(a_1, a_2, a_3, \dots)$ of $\A$ simply by $a$ and use the notation ${\rm d}a$ rather than ${\rm d}\mu_\A(a)$ for integrals over subsets of $\A$.  Recall that a function $f$ on $\A$ is said to be a simple Schwartz-Bruhat function on $\A$ if for each natural number $i$ there is an $f_i$ in $SB(\mathds Q_{p_i})$ such that \[f(a) = \prod_{i\in \mathds N} f_i(a_i)\] and for almost all $i$ the function $f_i$ is the vacuum vector $\Omega_i$.

\begin{proposition}\label{five:prop:pretheorem}
The space $SB(\A)$ is in the domain of the operator $\Delta_\A$ if and only if $\sigma$ is finite.
\end{proposition}

\begin{proof}

Since $\mathcal F_\A^{-1}\Omega_\A$ is equal to $\Omega_\A$ and $\mathcal F_\A$ is unitary on $L^2(\mathds \A)$, the function $\Omega_\A$ is in the domain of $\Delta_A$ if and only if it is in the domain of $\mathcal M$.  Calculate $\big|\big| {\mathcal M}\Omega_{\A}\big|\big|_{\A}^2$ to obtain the equality 
\begin{align}\label{sec4:mmupperandlower}
\big|\big| {\mathcal M}\Omega_{\A}\big|\big|_{\A}^2 & = \bigg|\bigg|\sum_{i\in \mathds N} \sigma_i|a_i|_{i}^b \prod_{i\in \mathds N} \Omega_i(a_i)\bigg|\bigg|_{\A}^2\notag\\
& = \int_{\A} \bigg(\sum_{i\in \mathds N} \sigma_i |a_i|_{i}^b\bigg)^2\Omega_{\A}(a)\,{\rm d}a\notag\\
& = \int_{\A} \Bigg(\sum_{i\in \mathds N} \sigma_i^2|a_i|_{i}^{2b} \Omega_{\A}(a) + \sum_{\substack{i,j \in \mathds N\\i\ne j}} \sigma_i\sigma_j|a_i|_{i}^b|a_j|_{j}^b\Omega_{\A}(a)\Bigg)\,{\rm d}a\notag\\
& = \sum_{i\in \mathds N} \sigma_i^2\int_{\A} |a_i|_{i}^{2b} \Omega_{\A}(a)\,{\rm d}a + \sum_{\substack{i,j \in \mathds N\\i\ne j}} \sigma_i\sigma_j\int_{\A} |a_i|_{i}^b|a_j|_{j}^b\Omega_{\A}(a)\,{\rm d}a\notag\\
& = \sum_{i\in \mathds N} \sigma_i^2\int_{\mathds Z_{p_i}} |a_i|_{i}^{2b} \,{\rm d}a_i + \sum_{\substack{i,j \in \mathds N\\i\ne j}} \sigma_i\sigma_j\int_{\mathds Z_{p_i}} |a_i|_{i}^b\,{\rm d}a_i \cdot \int_{\mathds Z_j}|a_j|_{j}^b\,{\rm d}a_j.
\end{align}
The lower bound of Lemma~\ref{five:lem:cestimate} and \eqref{sec4:mmupperandlower} together imply that 
\begin{align}\label{sec4:mmbelow}
\big|\big| {\mathcal M}\Omega_{\A}\big|\big|_{\A}^2 & > \sum_{i\ne j} \sigma_i\sigma_j\int_{\mathds Z_{p_i}} |a_i|_{i}^b\,{\rm d}a_j \cdot \int_{\mathds Z_j}|a_j|_{j}^b\,{\rm d}a_j\notag\\%
&>  \frac{1}{4}\sum_{i\ne j} \sigma_i\sigma_j > \frac{1}{4}m_1 \sum_{i>1} \sigma_i.
\end{align}
The upper bound of Lemma~\ref{five:lem:cestimate} and \eqref{sec4:mmupperandlower} together imply that
\begin{align}\label{sec4:mmabove}
\big|\big| {\mathcal M}\Omega_{\A}\big|\big|_{\A}^2 & < 2\sum_{i\in \mathds N} \sigma_i^2 + 4\sum_{\substack{i,j \in \mathds N\\i\ne j}} \sigma_i\sigma_j\notag\\ & < 4\Bigg(\sum_{i\in \mathds N} \sigma_i^2 + \sum_{\substack{i,j \in \mathds N\\i\ne j}} \sigma_i\sigma_j\Bigg) = 4\bigg(\sum_{i\in \mathds N} \sigma_i\bigg)^2.
\end{align}
If ${\mathcal M}\Omega_{\A}$ is square integrable, then \eqref{sec4:mmbelow} implies that $\sigma - \sigma_1$ is finite, hence $\sigma$ is finite.  The function $\Omega_\A$ is in $SB(\A)$ and so if $SB(\A)$ is in the domain of $\Delta_\A$, then $\sigma$ is finite.

If $\sigma$ is finite, then \eqref{sec4:mmabove} implies that ${\mathcal M}\Omega_{\A}$ is square integrable and so $\Omega_\A$ is in the domain of $\Delta_\A$.   To show that any simple Schwartz-Bruhat function on $\A$ is in the domain of $\Delta_\A$, it suffices to show that any simple Schwartz-Bruhat function on $\A$ is in the domain of $\mathcal M$.  If $f$ is a simple Schwartz-Bruhat function on $\A$, then there is a natural number $\ell$ such that for each $i$ in $\{1, \dots, \ell\}$, the function $f_i$ is in $SB(\mathds Q_{p_i})$ and for any $a$ in $\A$, \begin{equation}\label{4:prodforf}f(a) = \prod_{i\in \{1, \dots, \ell\}} f_i(a_i)\cdot \prod_{i> \ell}\Omega_i(a_i).\end{equation}  Use \eqref{4:prodforf} to obtain the equality \begin{align}\label{sec4:finitemuandfinitesquare}(\mathcal Mf)(x) &= \sum_{j\in\{1, \dots, \ell\}} \sigma_j|a_j|_j^bf_j(a_j)\prod_{\substack{i\in\{1, \dots, \ell\}\\i\ne j}} f_i(a_i)\cdot \prod_{i> \ell}\Omega_i(a_i) \notag\\&\hspace{2.5in}+ \sum_{j >  \ell} \sigma_j|a_j|_j^b\Omega_j(a_j)\prod_{i\in \{1, \dots, \ell\}} f_i(a_i)\cdot \prod_{\substack{i>\ell\\i\ne j}}\Omega_i(a_i).\end{align} Since for each $j$ in $\{1, \dots, \ell\}$ the function $g_j$ given by \[g_j(a_j) = |a_j|^bf(a_j)\] is square integrable, the square integrability of $\mathcal M \Omega_\A$ implies the square integrability of $f$. Finiteness of $\sigma$ therefore implies that simple Schwartz-Bruhat functions on $\A$ are in the domain of $\Delta_\A$.  Since finite sums of square integrable functions are square integrable, finiteness of $\sigma$ implies that any Schwartz-Bruhat function on $\A$ is in the domain of $\Delta_\A$.
\end{proof}

\begin{theorem}
The operator $\Delta_\A$ is essentially self adjoint on $SB(\A)$ if and only if $\sigma$ is finite.
\end{theorem}

\begin{proof}
If $\Delta_\A$ is essentially self adjoint on $SB(\A)$, then it is defined on $SB(\A)$ and Proposition~\ref{five:prop:pretheorem} implies that $\sigma$ is finite.

Suppose that $\sigma$ is finite.  Proposition \ref{five:prop:pretheorem} implies that $SB(\A)$ is in the domain of $\Delta_\A$ and so $\Delta_\A$ is a densely defined operator.  In this case, the operator $\Delta_\A$ is essentially self adjoint on $SB(\A)$ since it is the Fourier transform of a densely defined, real valued multiplication operator.
\end{proof}

Extend in the same way in which we defined the $p$-adic Vladimirov operator the adelic Laplacian $\DA$ to act on a function $f$ with \[f\colon \mathds R_+ \times \A \to \mathds R\] such that for each positive real $t$, $f(t,\cdot)$ is in the domain of $\DA$.  Differentiation with respect to the time variable of such a function $f$ is performed as it is in the $p$-adic setting.  In particular, if $a$ is in $\A$, then \[\dfrac{{\rm d}f}{{\rm d}t}(t,a) = \dfrac{{\rm d}f(\cdot, a)}{{\rm d}t}(t).\]  The adelic diffusion equation is the equation given by \begin{align}\label{eqAd} \frac{{\rm d}f}{{\rm d}t}(t,a) = -\Delta_\A f (t,a).\end{align}

\begin{theorem}\label{6:thm:pre7}
If $\rho$ is the fundamental solution to \eqref{eqAd}, then \begin{equation}\label{AdelicHeatFundSolution}\rho(t,a) = \big(\mathcal F^{-1}_{\A}\e^{-t\sum_i\sigma_i|\cdot|_i^b}\big)(a).\end{equation} For fixed positive $t$, the function $\rho(t, \cdot)$ is a probability density function on $\A$ that converges as $t$ tends to 0 from the right to the Dirac measure on $\mathds A$ that is concentrated at 0.  For every $x$ in $\A$, the function $\rho$ gives rise to a path measure $P_x^\A$ on $D(I\colon \A)$ that is conditioned so that paths almost surely start at $x$ at time 0.  When $x$ is equal to $(x_i)$, the measure $P_x^\A$ agrees with the product measure $\otimes_{i}P^{\i}_{x_i}$ on the adelic paths viewed as a subset of $D(I)$.
\end{theorem}

\begin{proof}
Take the adelic Fourier transform of both sides of \eqref{eqAd} to obtain the pseudo differential equation \[\dfrac{{\rm d}\mathcal F_\A f}{{\rm d}t}(t,a) = -\big(\mathcal M \mathcal F_\A f\big) (t,a) = -\bigg(\sum_{i\in\mathds N} \sigma_i|a_i|^b\bigg) \mathcal F_\A f(t,a),\] which has \begin{align}\label{Sec5:eq:inv}\big(\mathcal F_\A f\big)(t,a) = \e^{-t\sum_{i\in\mathds N} \sigma_i|a_i|^b}\end{align} as its only solution up to a constant multiple.  The inverse Fourier transform in the $a$ coordinate of the function on the right hand side of \eqref{Sec5:eq:inv} is the desired function $\rho$.

Take $a$ to be an element of $\A$ and let $a_i$ be the $i^{\rm th}$ component of $a$. The function \[\phi(t, \cdot) = e^{-t\sum_i\sigma_i|\cdot|_i^b}\] is formally an infinite product of functions on $L^2(\mathds Q_{p_i})$.  Since the sum of the $\sigma_i$ is finite, $\phi(t, \cdot)$ is the limit in the $L^2(\A)$ norm of a sequence of well-defined square integrable functions on $\mathds A$ and itself in $L^2(\A)$, implying that its inverse adelic Fourier transform is just the product of the inverse Fourier transforms of the $p_i$-adic components. The function $\rho$ is therefore given by the equality \[\rho(t,a) = \prod_{i\in\mathds N} \rho^i(t,a_i),\] where $\rho^i$ is a solution to \eqref{eq1}.  Since each $\rho^i(t, \cdot)$ is a probability density function, $\rho$ is positive as a product of positive functions and of total mass 1 as an integral over the adeles is a product of the integral of all of its $p$-adic components. The function $\rho$ is, therefore, a probability density function.

To show that $\rho(t,\cdot)$ converges in the weak-$\ast$ topology on the space of measures on $\mathds A$ to a Dirac measure on $\mathds A$ that is concentrated at 0, prove that for any bounded continuous function $f$ on $\mathds A$, \[\lim_{t\to 0^+} \int_\A f(a)\rho(t,a)\;{\rm d}a  = g(0).\] It suffices to show that if $B$ is an open subset of $\mathds A$ that contains 0, then \[\lim_{t\to 0^+} \int_B \rho(t,a)\;{\rm d}a  = 1.\]  Suppose to this end that $B$ is an open subset of $\mathds A$ that contains 0.  There is a natural number $N$ and bijections $j$ and $n$ from $\mathds N$ to itself such that $B$ contains an open set $B^\prime$ with \[B^\prime = \prod_{i\in\mathds N} O_i\] with the property that $O_i$ equals $\mathds Z_{p_i}$ if $i$ is not equal to $j_s$ with $s$ in $\{1, \dots, N\}$ and for each $s$ in $\{1, \dots, N\}$, $O_{j_s}$ is equal to $B^{j_s}_{-{n_s}}(0)$. If \begin{align}\label{diracmeasurepf}\lim_{t\to 0^+} \int_{B^\prime}\rho(t,a)\,{\rm d}a = 1,\end{align} then the same equality holds where the integral is taken over $B$.  Integrals of simple adelic functions are products of integrals over the $p$-adic components, which implies that \begin{align}\label{sec4:prodineqatend}\int_{B^\prime}\rho(t,x)\,{\rm d}a &= \prod_{s\in\mathds N}\int_{O_{j_s}}\rho^{\,j_s}\big(t,a_{j_s}\big)\,{\rm d}a_{j_s}\notag\\ & = \prod_{s\leq N}\int_{O_{j_s}}\rho^{\,j_s}\big(t,a_{j_s}\big)\,{\rm d}a_{j_s} \cdot \prod_{s > N}\int_{O_{j_s}}\rho^{\,j_s}\big(t,a_{j_s}\big)\,{\rm d}a_{j_s}.\end{align}

A straightforward modification of the arguments in \cite{var97} shows that, for positive $t$, the density function $\rho^i(t,x)$ for the random variable $X_t^{\i}$ satisfies the equality
\begin{align}\label{sec3:formulaforf}\rho^i(t,x_i) &= \sum_{r\in \mathds Z} e^{-\sigma tp_i^rb}\int_{S_r} \mathds \chi(x_iy_i)\,{\rm d}y_i\notag\\ &= \sum_{r\in \mathds Z} \Big(\e^{-\sigma tp_i^{rb}} - \e^{-\sigma tp_i^{(r+1)b}}\Big)\int_{B^i_r} \mathds \chi(xy)\,{\rm d}y\notag\\& = \sum_{r\in\mathds Z}\Big(\e^{-\sigma t p_i^{rb}} - \e^{-\sigma  t p_i^{(r+1)b}}\Big)p_i^r{\mathds 1}_{p_i^{-r}}(x_i).\end{align} Suppose that $\nu$ is an integer and use equation \eqref{sec3:formulaforf} for $\rho^i$ to see that \begin{align}\label{nuiszero}\int_{|a_i|\leq p_i^\nu}\rho^i\big(t,a_i\big)\,{\rm d}a_i &=  \int_{B^i_\nu} \rho(t,x_i)\,{\rm d}x_i\notag\\ &= \sum_{r\in \mathds Z} \Big(\e^{-\sigma_i tp_i^{rb}} - \e^{-\sigma_i tp_i^{(r+1)b}}\Big)p_i^r\int_{B^i_\nu} \mathds 1_{p_i^{-r}}(x)\,{\rm d}x_i\notag\\& = \e^{-\sigma_i tp_i^{-\nu b}} + \sum_{r\leq -\nu-1}p_i^{\nu+r}\Big(\e^{-\sigma_i tp_i^{rb}} - \e^{-\sigma_i tp_i^{(r+1)b}}\Big)\ge \e^{-\sigma_itp_i^{-\nu b}}.\end{align} Set $\nu$ to be equal to 0 in \eqref{nuiszero} to obtain the inequality \begin{align}\label{nuisnowzero}\int_{\mathds Z_{p_i}} \rho^i(t,a_i)\,{\rm d}a_i \ge  \e^{-\sigma_it}.\end{align} Inequality \eqref{nuisnowzero} together with \eqref{sec4:prodineqatend} implies that \begin{align}\label{sec4:prodineqatendd}\int_{B^\prime}\rho(t,a)\,{\rm d}a & > \prod_{s\leq N} \e^{-m_{j_s}tp_{j_s}^{-n_sb}}\cdot \prod_{s > N} \e^{-m_{j_s}t}\notag\\& \ge \e^{-\sigma t}\prod_{s\leq N} \e^{-m_{j_s}tp_{j_s}^{-n_sb}}.\end{align} The finiteness of $\sigma$ implies that \[\lim_{t\to 0^+} \e^{-\sigma t} = 1.\] The second term in the product \eqref{sec4:prodineqatendd} tends to 1 as $t$ tends to 0 from the right as the finite product of terms that have this property.  The integral $\int_{B^\prime}\rho(t,a)\,{\rm d}a$ is bounded above by 1, implying \eqref{diracmeasurepf}.

A \emph{restricted history} of the space of all paths on $\A$ is a history such that each place of a route of $h$ is a product of $p$-adic balls that are almost always the ring of integers in the respective $p$-adic component.  Suppose that $h$ is a restricted history of the space of all paths on $\A$ with epoch $e(h)$ and route $U(h)$ where \[U(h) = \big(U(h)_0, \dots, U(h)_{\ell(h)}\big).\]  If $U(h)_0$ does not contain 0, then define $P^{\A}(C(h))$ as equal to 0.  If $U(h)_0$ equals $\{0\}$, then define $P^{\A}(C(h))$ by \begin{align}\label{measadelica}P^{\A}(C(h)) &= \int_{U(h)_1} \cdots \int_{U(h)_{\ell(h)}} \rho\big(e(h)_1, z^1\big) \rho\big(e(h)_2 - e(h)_1, z^2 - z^1\big)\notag\\&\qquad\qquad\qquad\qquad\qquad\cdots \rho\big(e(h)_{\ell(h)} - e(h)_{\ell(h)-1}, z^{\ell(h)} - z^{\ell(h)-1}\big) \,{\rm d}z^{\ell(h)}\cdots \,{\rm d}z^1.\end{align} Denote by $U^i(h)_j$ the $i^{\rm th}$ component of the $j^{\rm th}$ place of the route of $h$ and by $U^i(h)$ the route \[U^i(h) = \big(U^i(h)_0, \dots, U^i(h)_{\ell(h)}\big)\] whose entries are subsets of $\mathds Q_{p_i}$.  Denote by $C(h)_i$ the cylinder set with the same epoch $e(h)$ and whose route is $U^i(h)$. For each $i$ and $j$, the set $U^i(h)_j$ is a ball in $\mathds Q_{p_i}$.  If $z^j$ is in $\A$, then denote by $z_i^j$ the $i^{\rm th}$ component of $z^j$.  Use this notation and the fact that integrals of products of simple adelic functions are products of the integrals of the $p$-adic factors to simplify \eqref{measadelica} and obtain
\begin{align*}
P^{\A}(C(h))
&= \int_{U(h)_1} \cdots \int_{U(h)_{\ell(n)}} \bigg(\prod_{i\in\mathds N} \rho^i\big(e(h)_1, z^1\big) \rho^i\big(e(h)_2 - e(h)_1, z^2 - z^1\big)\\&\qquad\qquad\qquad\qquad\qquad\cdots \rho^i\big(e(h)_{\ell(h)} - e(h)_{\ell(h)-1}, z^{\ell(h)} - z^{\ell(h)-1}\big)\bigg) \,{\rm d}z^{\ell(h)}\cdots \,{\rm d}z^1\\
&= \prod_{i\in\mathds N}\int_{U^i(h)_1} \cdots \int_{U^i(h)_{\ell(h)}} \rho^i\big(e(h)_1, z_i^1\big) \rho^i\big(e(h)_2 - e(h)_1, z_i^2 - z_i^1\big)\\&\qquad\qquad\qquad\qquad\qquad\cdots \rho^i\big(e(h)_{\ell(h)} - e(h)_{\ell(h)-1}, z_i^{\ell(h)} - z_i^{\ell(h)-1}\big) \,{\rm d}z_i^{\ell(h)}\cdots \,{\rm d}z_i^1\\
&= \prod_{i\in\mathds N} P^{i}(C(h)_i),
\end{align*}
hence \[P^{\A}(C(h)) = \prod_{i\in\mathds N} P^{i}(C(h)_i).\]  The measure $P^{\A}$ therefore agrees on the simple cylinder sets of $D(I\colon \A)$ with restricted histories with the product measure over the $p$-adic components.  Since these subsets of $D(I\colon \A)$ form a $\pi$-system that generates the cylinder sets, the two measures agree on the cylinder sets, \cite{bil1}. For any $P^\A_x$ measurable subset $A$ of paths, the probability $P^{\A}_x(A)$ is equal to $P^{\A}(A-x)$ and for each $i$ the measure $P^i_{x_i}(A^i)$ is equal to $P^i(A^i -x_i)$, where $A^i$ is the set of paths given by the $i^{\rm th}$ components of paths of $A$.  The equality of the measures concentrated on the paths initially at 0 implies the more general statement of equality for measures concentrated on the paths starting at an arbitrary $x$ in $\A$.
\end{proof}

\begin{theorem}\label{6:thm:7}
Suppose that $x$  and $y$ are in $\A$, that $\sigma$ is finite, and that $I$ is a time interval containing the positive real number $t$. Denote by $P^\A_{t, x, y}$ the measure $P^\A_x$ on $D(I\colon \A)$ conditioned to give full measure to the paths that take value $y$ at time $t$.  For each natural number $i$, denote by $P^{i}_{t, x_i,y_i}$ the measure $P^{i}_{x_i}$ conditioned to give full measure to the set of paths in $D\big(I\colon \mathds Q_{p_i}\big)$ that take value $y_i$ at time $t$.  The product measure given by $\otimes_{i\in\mathds N}P^{i}_{t, x_i, y_i}$ gives full measure to the set  $D(I\colon \A)$ viewed as a subset of the product space.  Furthermore, the measures $P^\A_{t, x, y}$ and $\otimes_{i\in\mathds N}P^{i}_{t, x_i, y_i}$ agree on $D(I\colon \A)$.
\end{theorem}

\begin{proof}

  Proposition~\ref{three:ExitCondInequal} implies that if $t$ is a non-negative real number, then \begin{align*}P^{i}_{t, x_i, y_i}\big(||X||_t\leq 1\big) &= P^{i}_{x_i}\big(||X||_t\leq 1\big|X_t = y_i\big) \\&\ge P^{i}_{x_i}\big(||X||_t\leq 1\big) = {\rm e}^{-\sigma_i\alpha_i tp_i^{-b}}.\end{align*}  Use the above estimate rather than Proposition~\ref{three:ExitEqual} and follow the argument used in the proof of Theorem~\ref{DAfullMeas} to see that the conditioned measures give the adelic paths full measure.

For each natural number $n$, let $\varepsilon^{(n)}$ be a sequence $\big(\varepsilon^{(n)}_i\big)$, almost all of which are equal to 1.  To say that $\varepsilon^{(n)}$ tends to 0 means that each $\varepsilon^{(n)}_i$ tends to 0 in $n$ but for each $n$ almost all of the $\varepsilon^{(n)}_i$ are equal to 1.  If $y$ is in $\mathds A$, denote by $B_{\varepsilon^{(n)}}(y)$ the set \[B_{\varepsilon^{(n)}}(y) = \prod_iB_{\varepsilon^{(n)}_i}(y_i),\] where $B_{\varepsilon^{(n)}_i}(y_i)$ is the ball \[B_{\varepsilon^{(n)}_i}(y_i) = \big\{z\in \mathds Q_{p_i}\colon |z-y_i| \leq\varepsilon^{(n)}_i\big\}.\]%

Suppose that $h$ is a restricted history and that $\varepsilon^{(n)}$ tends to 0.  There is an $M$ in $\mathds N$ so that $n$ is larger than $M$ implies that for each natural number $i$, $B_{\varepsilon^{(n)}_i}(0)$ is a subset of $\mathds Z_{p_i}$.  Proposition~\ref{three:ExitEqual} and Proposition~\ref{three:ExitCondInequal} together imply that if $n$ is larger than $M$, then \[P^{i}_{x_i}\Big(\bigcap_{1\leq j\leq \ell(h)}X^i_{e(h)_j}\in \mathds Z_{p_i}\Big| X^i_t\in y^i+ B_{\varepsilon^{(n)}_i}(0)\Big) > \e^{-t\alpha_i\sigma_i}.\] Therefore, if $n$ is larger than $M$, then for each positive real $\varepsilon$ there is a natural number $N$ so that
\begin{align}\label{sec4:eqaforthm47}&\prod_{i>N} P^{i}_{x_i}\Big(\bigcap_{1\leq j\leq \ell(h)}X^i_{e(h)_j}\in \mathds Z_{p_i}\Big|X^i_t \in y^i+ B_{\varepsilon^{(n)}_i}(0)\Big)\notag\\&\hspace{2in}> \prod_{i>N} P^{i}_{x_i}\Big(\bigcap_{1\leq j\leq \ell(h)}||X^i||_t \leq 1\Big)\notag\\&\hspace{2in} > \e^{-t\sum_{i>N}\alpha_i\sigma_i} \notag\\&\hspace{2in} > 1-\varepsilon\end{align}
and
\begin{align}\label{sec4:eqbforthm47bb}
\prod_{i>N} P^{i}_{t, x_i, y_i}\Big(\bigcap_{1\leq j\leq \ell(h)}X^i_{{e(h)_j}}\in \mathds Z_{p_i}\Big)& >\prod_{i>N} P^{i}_{t, x_i, y_i}\big(||X^i||_{t}\leq 1\big) \notag\\&> \e^{-t\sum_{i>N}\alpha_i\sigma_i} \notag\\&> 1-\varepsilon, 
\end{align}  
Inequalities \eqref{sec4:eqaforthm47} and \eqref{sec4:eqbforthm47bb} together imply that 
\begin{align}\label{4:eq:conditestforproducts}&\bigg|\prod_{i>N} P^{i}_{x_i}\Big(\bigcap_{1\leq j\leq \ell(h)}X^i_t\in \mathds Z_{p_i}\big|X^i_{e(h)_j} \in y^i+ B_{\varepsilon^{(n)}_i}(0)\Big)\notag\\&\hspace{3in} - \prod_{i>N} P^{i}_{t, x_i, y_i}\Big(\bigcap_{1\leq j\leq \ell(h)}X^i_t\in \mathds Z_{p_i}\Big)\bigg| < \varepsilon.
\end{align}

For any simple cylinder set associated to a restricted history $h$ and for any $x$ and $y$ in $\mathds A$ and positive real $t$, the probability measure $P^{\A}_{t, x, y}$ on adelic paths conditioned to start at $x$ at time 0 and to be at $y$ at time $t$ is given by 
\begin{align}\label{measadelic}P^{\A}_{t, x, y}(C(h)) &= \lim_{\varepsilon^{(n)} \to 0} P^{\A}_{x}\big(C(h)\big|X_t \in y+ B_{\varepsilon^{(n)}}(0)\big)\notag\\%
& =  \lim_{\varepsilon^{(n)} \to 0} \prod_i P^{i}_{x_i}\Big(C(h)_i\Big|X^i_t \in y^i+ B_{\varepsilon^{(n)}_i}(0)\Big)\notag\\%
& =  \lim_{\varepsilon^{(n)} \to 0}\Bigg( \prod_{1\leq i\leq N} P^{i}_{x_i}\Big(C(h)_i\Big|X^i_t \in y^i+ B_{\varepsilon^{(n)}_i}(0)\Big)\notag\\&\qquad\qquad\cdot \prod_{i>N} P^{i}_{x_i}\Big(\bigcap_{1\leq j\leq \ell(h)}X^i_t\in \mathds Z_{p_i}\big|X^i_{e(h)_j} \in y^i+ B_{\varepsilon^{(n)}_i}(0)\Big)\Bigg)\notag\\%
& =  \lim_{\varepsilon^{(n)} \to 0}\prod_{1\leq i\leq N} P^{i}_{x_i}\big(C(h)_i\big|X^i_t \in y^i+ B_{\varepsilon^{(n)}_i}(0)\big)\notag\\&\qquad\qquad\cdot\lim_{\varepsilon^{(n)} \to 0} \prod_{i>N} P^{i}_{x_i}\Big(\bigcap_{1\leq j\leq \ell(h)}X^i_t\in \mathds Z_{p_i}\big|X^i_{e(h)_j} \in y^i+ B_{\varepsilon^{(n)}_i}(0)\Big)\notag\\%
& =  \prod_{1\leq i\leq N} P^{i}_{t, x_i, y_i}\big(C(h)_i\big)\\&\qquad\qquad\cdot\lim_{\varepsilon^{(n)} \to 0} \prod_{i>N} P^{i}_{x_i}\Big(\bigcap_{1\leq j\leq \ell(h)}X^i_t\in \mathds Z_{p_i}\big|X^i_{e(h)_j} \in y^i+ B_{\varepsilon^{(n)}_i}(0)\Big).\notag
\end{align} 
  
 Take the absolute value of the difference between $P^{\A}_{t, x, y}(C(h))$ and $\prod_{i\in\mathds N} P^{i}_{x_i, y_i, t}(C(h)_i)$ and use \eqref{4:eq:conditestforproducts} and \eqref{measadelic} to obtain the estimate
\begin{align*}
&\Big|\prod_{i\in\mathds N} P^{i}_{t, x_i, y_i}(C(h)_i) - P^{\A}_{t,x,y}(C(h)) \Big|\\
& \qquad =  \Big| \prod_{i\in \mathds N} P^{i}_{t, x_i, y_i}\big(C(h)_i\big) - \lim_{\varepsilon^{(n)} \to 0}\prod_{i\in\mathds N} P^{i}_{x_i}\Big(\bigcap_{1\leq j\leq \ell(h)}X^i_t\in \mathds Z_{p_i}\big|X^i_{e(h)_j} \in y^i+ B_{\varepsilon^{(n)}_i}(0)\Big)\Big| 
\\& \qquad =  \Big|\prod_{1\leq i\leq N} P^{i}_{t, x_i, y_i}\big(C(h)_i\big)\Big|\cdot \lim_{\varepsilon^{(n)} \to 0}\Big| \prod_{i>N} P^{i}_{x_i}\Big(\bigcap_{1\leq j\leq \ell(h)}X^i_t\in \mathds Z_{p_i}\big|X^i_{e(h)_j} \in y^i+ B_{\varepsilon^{(n)}_i}(0)\Big) \\&\hspace{4.9in}
 - \prod_{i>N} P^{i}_{t, x_i, y_i}\big(C(h)_i\big)\Big|
\\& \qquad \leq  \Big|\prod_{1\leq i\leq N} P^{i}_{t, x_i, y_i}\big(C(h)_i\big)\Big|\cdot \lim_{\varepsilon^{(n)} \to 0}\varepsilon \leq \varepsilon.
\end{align*}
Since $\varepsilon$ was arbitrarily chosen, \[P^{\A}_{t, x, y}(C(h)) = \prod_{i\in\mathds N} P^{i}_{t, x_i, y_i}(C(h)_i) = \big(\!\otimes_{i}P^{i}_{t, x_i, y_i}\big)(C(h)).\]  The measure $P^{\A}_{t, x, y}$ therefore agrees on the simple cylinder sets of $D(I\colon \A)$ with restricted histories with the measure induced on the adelic paths by the product measure.  Since these sets form a $\pi$-system that generates cylinder sets, the two measures agree on the cylinder sets, \cite{bil1}. 
\end{proof}


\section{Adelic Path Integrals and Propagators}\label{five}

Suppose throughout this section that $v$ is a real valued, non-negative, bounded, continuous function on $\A$ and refer to such a function as a \emph{potential function}.  If $\psi$ is a Schwartz-Bruhat function, then define for each $a$ in $\A$ the function $V\psi$ by \[(V\psi)(a) = v(a)\psi(a).\]  The operator $V$ maps a Schwartz-Bruhat function $\psi$ to the function $V\psi$ and is a \emph{potential} associated to the potential function $v$.  Retain throughout this section the notation where a function denoted by a lowercase letter corresponds to a multiplication operator that is denoted by the corresponding uppercase letter.  Denote respectively by $H_{\A}^0$ and $H_{\A}$ the \emph{free Schr\"{o}dinger operator} and \emph{Schr\"{o}dinger operator with potential} $V$ acting on $SB(\A)$ and given by \[H_{\A}^0 = \Delta_{\A} \quad {\rm and}\quad H_{\A} = \Delta_{\A}+ V.\] The operator $H_{\A}$ is essentially self adjoint on the Schwartz-Bruhat functions as the sum of an essentially self adjoint operator on this domain and a bounded real valued multiplication operator.  To compress notation, henceforth denote respectively by $D_\A[0, \infty)$ and, for all positive $t$, by $D_\A[0,t]$ the path spaces $D([0,\infty): \A)$ and $D([0,t]: \A)$. Denote by $\pi_t^0$ and by $\pi_t$ the operators on $L^2(\A)$ that act on any Schwartz-Bruhat function $\alpha$ by \begin{equation}\label{semigroup:defs}\big(\pi_t^0\alpha\big)(x) = \int_{D_\A[0,\infty)}\alpha(\omega(t))\,{\rm d}P^\A_x\hspace{.125in} {\rm and}\hspace{.125in} (\pi_t\alpha)(x) = \int_{D_\A[0,\infty)} \e^{-\int_0^t v(\omega(s))\,{\rm d}s} \alpha(\omega(t))\,{\rm d}P^\A_x.\end{equation} Notice that a sample path $\omega$ has at most countably many discontinuities as a c\'{a}dl\'{a}g function and $v$ is bounded and continuous, implying that $v\circ \omega$ is a bounded function on $[0,\infty)$ with at most countably many discontinuities.  The integral appearing in the exponential of the right hand side of the equality \eqref{semigroup:defs} is, therefore, well defined as a Riemann integral.

\begin{lemma}\label{5:FK:a}
For each positive real number $t$, $\pi_t$ extends to a bounded self adjoint operator defined on all of $L^2(\A)$.
\end{lemma}

\begin{proof}
Fix $t$ larger than 0 and suppose that $\alpha$ is in $L^2(\A)$.  Condition on the event that a path is at $y$ at time $t$ to obtain the equalities
\begin{align*}
(\pi_t\alpha)(x) &= \int_{D_\A[0,\infty)} \e^{-\int_0^t v(\omega(s))\,{\rm d}s} \alpha(\omega(t))\,{\rm d}P^\A_x\\
&=\int_{\A}\bigg\{\int_{D_\A[0,\infty)} \e^{-\int_0^t v(\omega(s))\,{\rm d}s} \alpha(\omega(t))\,{\rm d}P^\A_{t,x,y}\bigg\}\rho_t(x-y)\,{\rm d}{y}\\
&=\int_{\A}\bigg\{\int_{D_\A[0,\infty)} \e^{-\int_0^t v(\omega(s))\,{\rm d}s} \alpha(y)\,{\rm d}P^\A_{t,x,y}\bigg\}\rho(t,x-y)\,{\rm d}{y}\\
&=\int_{\A}\bigg\{\int_{D_\A[0,\infty)} \e^{-\int_0^t v(\omega(s))\,{\rm d}s}\,{\rm d}P^\A_{t,x,y}\bigg\}\rho(t,x-y)\alpha(y)\,{\rm d}{y} = \int_\A K_t(x,y)\alpha(y)\,{\rm d}{y}\\
\end{align*}
where \[K_t(x,y) = \bigg\{\int_{D_\A[0,\infty)} \e^{-\int_0^t v(\omega(s))\,{\rm d}s}\,{\rm d}P^\A_{t,x,y}\bigg\}\rho(t,x-y).\]  The operator $\pi_t$ is an integral operator with kernel $K_t$.  The kernel $K_t$ depends only on the values that a path takes on the interval $[0,t]$ and so \begin{equation}\label{sec5:Ktgoesfinite} K_t(x,y) = \bigg\{\int_{D_\A[0,t]} \e^{-\int_0^t v(\omega(s))\,{\rm d}s}\,{\rm d}P^\A_{t,x,y}\bigg\}\rho(t,x-y).\end{equation} Since $K_t$ is a bounded real valued function of $(x, y)$, in order to show that $\pi_t$ is bounded and self adjoint, it suffices to show that $K_t$ is symmetric.  Define time reflection, $\ast$, to act on the probability space $(D_\A[0,t], P_{t,x,y})$ in the following way.  For each $\omega$ in $D_\A[0,t]$, define $\omega^\ast$ at a time point $s$ in $[0,t)$ by \[\omega^\ast(s) = \omega(t - s - 0)\quad {\rm and}\quad \omega^\ast(t) = \omega(0).\]  For each measurable subset $E$ of $D_\A[0,t]$, define \[Q^\A_{t,x,y}(E) = P^\A_{t, x, y}(E^\ast).\] Time reflection is an involution mapping the $\sigma$-algebra of cylinder sets to itself.  For any fixed $s$ in $[0,t]$, Skorokhod paths are left continuous at $s$ with full probability in the probability space $\big(D_\A[0,t], P^\A_{t,x,y}\big)$, implying that $Q^\A_{t,x,y}$ and $P^\A_{t, y, x}$ agree on the cylinder sets of finite type and therefore on all cylinder sets since the cylinder sets of finite type form a $\pi$-system that generates the $\sigma$-algebra of cylinder sets.  The integral of the exponential term is invariant under time reversal and $\rho(t,x-y)$ equals $\rho(t,y-x)$, implying that
\begin{align*}
K_t(x,y) &= \bigg\{\int_{D_\A[0,t]} \e^{-\int_0^t v(\omega^\ast(s))\,{\rm d}s}\,{\rm d}P^\A_{t,x,y}(\omega^\ast)\bigg\}\rho(t,x-y)\\
&= \bigg\{\int_{D_\A[0,t]} \e^{-\int_0^t v(\omega(s))\,{\rm d}s}\,{\rm d}P^\A_{t,x,y}(\omega^\ast)\bigg\}\rho(t,x-y)\\%
&=\bigg\{\int_{D_\A[0,t]} \e^{-\int_0^t v(\omega(s))\,{\rm d}s}\,{\rm d}Q^\A_{t,x,y}(\omega)\bigg\}\rho(t,x-y)\\
&=\bigg\{\int_{D_\A[0,t]} \e^{-\int_0^t v(\omega(s))\,{\rm d}s}\,{\rm d}P^\A_{t,y,x}(\omega)\bigg\}\rho(t,y-x) = K_t(y,x).
\end{align*}

\end{proof}

\begin{lemma}\label{5:FK:b}
The one parameter family of operators $(\pi_t)_{t> 0}$ forms a strongly continuous semigroup.
\end{lemma}

\begin{proof}
Suppose that $\omega$ is in $D_\A[0,\infty)$ and that $\alpha$ is in $SB(\A)$.  For each positive $s$, denote by $\omega_s$ the function that maps $t$ to $\omega(s+t)$.  The law of total probability implies that 
\begin{align*}
\big(\pi_{s+t}\alpha\big)(x) &= \int_{D_\A[0,\infty)}\e^{-\int_0^{t+s}v(\omega(u))\,{\rm d}u}\alpha(\omega(t+s))\,{\rm d}P^\A_x(\omega)\\
&= \int_{D_\A[0,\infty)}\e^{-\int_0^{s}v(\omega(u))\,{\rm d}u}\e^{-\int_s^{t+s}v(\omega(r))\,{\rm d}r}\alpha(\omega(t+s))\,{\rm d}P^\A_x(\omega)\\
&= \int_\A\bigg\{\int_{D_\A[0,\infty)} \e^{-\int_0^{s}v(\omega(u))\,{\rm d}u}\e^{-\int_s^{t+s}v(\omega(r))\,{\rm d}r}\alpha(\omega_s(t))\,{\rm d}P^\A_{x,y,s}(\omega)\bigg\}\rho(s,y-x)\,{\rm d}y\\
& = \int_\A \bigg\{\int_{D_\A[0,\infty)}\e^{-\int_0^s v(\omega(u))\,{\rm d}u}\e^{-\int_0^tv(\omega_s(r))\,{\rm d}r}\alpha(\omega_s(t))\,{\rm d}P^\A_{x,y,s}(\omega)\bigg\}\rho(s,y-x){\rm d}y\\
& = \int_\A \bigg\{\int_{D_\A[0,\infty)}\e^{-\int_0^s v(\nu(u))\,{\rm d}u}\\&\qquad\qquad\bigg\{\int_{D_\A [0,\infty)}\e^{-\int_0^tv(\omega(r)+y)\,{\rm d}r}\alpha(\omega(t) +y)\,{\rm d}P^\A(\omega)\bigg\}\,{\rm d}P^\A_{x,y,s}(\nu)\bigg\}\rho(s,y-x){\rm d}y\\
& = \int_\A \bigg\{\int_{D_\A[0,\infty)}\e^{-\int_0^s v(\nu(u))\,{\rm d}u}\\&\qquad\qquad\bigg\{\int_{D_\A[0,\infty)}\e^{-\int_0^tv(\omega(r))\,{\rm d}r}\alpha(\omega(t))\,{\rm d}P^\A_{\nu(s)}(\omega)\bigg\}\,{\rm d}P^\A_{x,y,s}(\nu)\bigg\}\rho(s,y-x)\,{\rm d}y\\
& = \int_{D_\A[0,\infty)} \e^{-\int_0^s v(\nu(u))\,{\rm d}u} \bigg\{\int_{D_\A[0,\infty)}\e^{-\int_0^{t}v(\omega(r))\,{\rm d}r}\alpha(\omega(t))\,{\rm d}P^\A_{\nu(s)}(\omega)\bigg\}\,{\rm d}P^\A_x(\nu)\\ &= \big(\pi_s\pi_t\alpha\big)\!(x),
\end{align*}
hence $\big(\pi_{t}\big)_{t>0}$ is a semigroup.

Utilize below the ``big O'' notation below for the $t$ dependent error term below that is uniform in $x$.  If $\alpha$ is $SB(\A)$, then
\begin{align*}
(\pi_t\alpha - \alpha)(x) &= \int_{D_\A[0,\infty)}\e^{-\int_0^tv(\omega(s))\,{\rm d}s}\alpha(\omega(t))\,{\rm d}P^\A_x - \alpha(x)\\
&= \int_{D_\A[0,\infty)}(1 - O(t))\alpha(\omega(t))\,{\rm d}P^\A_x - \alpha(x)\\
&= \int_{D_\A[0,\infty)}\alpha(\omega(t))\,{\rm d}P^\A_x - \alpha(x) - \int_{D_\A[0,\infty)}O(t)\alpha(\omega(t))\,{\rm d}P^\A_x - \alpha(x)\\
&=  (\pi^0_t\alpha)(x) - \alpha(x) - O(t)(f_t\ast \alpha)(x).
\end{align*}
Take the norm of the difference $\pi_t\alpha - \alpha$ to obtain the bound \[||\pi_t\alpha - \alpha||_2 \leq ||\pi_t^0\alpha - \alpha||_2 + O(t)||f_t\ast \alpha||_2 \to 0\] as $t$ tends to 0 from the right, which implies that \begin{equation}\label{sec4:almostsemigroup}\lim_{t\to 0^+}||\pi_t\alpha - \alpha||_2 = 0.\end{equation}

For each positive $t$, Young's convolution inequality implies that the operator $\pi_t$ is bounded on $SB(\A)$ by 1 and so uniquely extends to an operator, again denoted by $\pi_t$, on all of $L^2(\A)$.  Furthermore, if $\varepsilon$ is a positive real number and $\beta$ is in $L^2(\A)$, then there is an $\alpha$ in $SB(\A)$ such that \[||\beta - \alpha||_2 < \varepsilon,\] and so %
\begin{align*}
||\pi_t\beta - \beta||_2 &= ||\pi_t(\beta - \alpha + \alpha) - (\beta - \alpha +\alpha)||_2\\
&\leq ||\pi_t(\beta - \alpha)||_2 + ||\beta - \alpha||_2 + ||\pi_t\alpha  -\alpha||_2
\\&\leq 2||\beta - \alpha||_2 + ||\pi_t\alpha  -\alpha||_2 \to 2||\beta - \alpha||_2 <2\varepsilon
\end{align*}
as $t$ tends to 0 from the right.  Since $\varepsilon$ was arbitrarily chosen, $\pi_t$ is strongly continuous on $L^2(\A)$.  A similar argument using approximation of square integrable functions by Schwartz-Bruhat function shows that $(\pi_t)_{t> 0}$ is a semi-group on $L^2(\A)$.  In particular, reuse the notation for $\alpha$ and $\beta$ to obtain for any positive real numbers $t$ and $s$ the estimate
\begin{align*}
||\pi_{t+s}\beta - \pi_t\pi_s\beta||_2 &= ||\pi_{t+s}\alpha -\pi_t\pi_s\alpha + \pi_{t+s}(\beta - \alpha) - \pi_t\pi_s(\beta - \alpha)||_2\\
&\leq ||\pi_{t+s}\alpha -\pi_t\pi_s\alpha||_2 + ||\pi_{t+s}(\beta - \alpha)||_2 + ||\pi_t\pi_s(\beta - \alpha)||_2\\
&= ||\pi_{t+s}(\beta - \alpha)||_2 + ||\pi_t\pi_s(\beta - \alpha)||_2 <  2||\beta - \alpha||_2 <\varepsilon.
\end{align*}

\end{proof}

\begin{lemma}\label{5:FK:c}
The infinitesimal generator for $(\pi_t)_{t> 0}$ is $-H_\A$.
\end{lemma}

\begin{proof}
Suppose that $\alpha$ is in $SB(\A)$.  The law of total probability implies that 
\begin{align*}
(\pi^0_t\alpha)(x) &= \int_{D_\A[0,\infty)}\alpha(\omega(t))\,{\rm d}P^\A_x(\omega)\\
&= \int_{\A}\bigg\{\int_{D_\A[0,\infty)}\alpha(\omega(t))\,{\rm d}P^\A_{t,x,y}(\omega)\bigg\}\rho(t,x-y)\,{\rm d}y\\
&= \int_{\A}\alpha(y)\rho(t,x-y)\,{\rm d}y = (\rho(t,\cdot)\ast \alpha)(x).
\end{align*}

Utilize the ``little o'' notation.  There is a function $R_t$ on $\A$ such that $|R_t(\cdot)|$ equals $o(t)$ uniformly in the variable in $\A$ and
\begin{align*}
\FA(\rho(t, \cdot)\ast \alpha) & = \big(\e^{-t\mathcal M}\big)\FA(\alpha)\\
& = (1 - t\mathcal M + R_t)\FA(\alpha)\\
& = \FA(\alpha)  - t(\mathcal M\FA)(\alpha)  + (R_t\FA)(\alpha).\\
\end{align*}%
Take the inverse Fourier transform to obtain the equality %
\begin{align*}
\rho(t, \cdot)\ast \alpha & = \alpha - t\Delta_\A\alpha + (\FA^{-1}R_t)\ast \alpha,
\end{align*}
where \begin{equation}\label{sec5:o1}\big|\big|\big(\FA^{-1}R_t\big)\ast \alpha\big|\big| = o(t)||\alpha||.\end{equation}%
There is a function $E(t)$ on $\A$ such that \begin{equation}\label{sec5:o2}|E(t)|  = o(t)\end{equation} holds uniformly in the starting point $x$ of the path $\omega$ and \begin{equation}\label{sec5:o3}\e^{-\int_0^tv(\omega(s))\,{\rm d} s}= 1 - \int_0^tv(\omega(s))\,{\rm d}s + E(t).\end{equation}

Equation \eqref{sec5:o3} implies that
\begin{align}\label{sec5:o4}
(\pi_t\alpha)(x) &= \int_{D_\A[0,\infty)} \e^{-\int_0^tv(\omega(s))\,{\rm d}s}\,\alpha(\omega(t))\,{\rm d}P^\A_x\notag\\
&= \int_{D_\A[0,\infty)} \bigg(1 - \int_0^tv(\omega(s))\,{\rm d}s + E(t)\bigg) \, \alpha(\omega(t))\,{\rm d}P^\A_x\notag\\
&= \int_{D_\A[0,\infty)} \alpha(\omega(t))\,{\rm d}P^\A_x - \int_{D_\A[0,\infty)} \bigg(\int_0^tv(\omega(s))\,{\rm d}s\bigg)\alpha(\omega(t))\,{\rm d}P^\A_x  \notag\\& \qquad\qquad\qquad\qquad\qquad\qquad\qquad\qquad\qquad\qquad+ \int_{D_\A[0,\infty)} E(t)\, \alpha(\omega(t))\,{\rm d}P^\A_x.
\end{align}
Expand the following difference quotient using \eqref{sec5:o4} to obtain the equalities
\begin{align*}
\dfrac{(\pi_t\alpha)(x) - \alpha(x)}{t} &= \frac{1}{t}\bigg\{\int_{D_\A[0,\infty)} \alpha(\omega(t))\,{\rm d}P^\A_x \\&\qquad- \int_{D_\A[0,\infty)} \bigg(\int_0^tv(\omega(s))\,{\rm d}s\bigg)\alpha(\omega(t))\,{\rm d}P^\A_x  + \int_{D_\A[0,\infty)} E(t) \alpha(\omega(t))\,{\rm d}P^\A_x\bigg\}\\
&= \frac{1}{t}\bigg\{-t(\Delta_\A\alpha)(x) + \big(\big(\FA^{-1} R(t)\big)\ast \alpha\big)\!(x) \phantom{\int_{D_\A[0,\infty)}}\\&\qquad- \int_{D_\A[0,\infty)} \bigg(\int_0^tv(\omega(s))\,{\rm d}s\bigg)\alpha(\omega(t))\,{\rm d}P^\A_x  + \int_{D_\A[0,\infty)} E(t) \alpha(\omega(t))\,{\rm d}P^\A_x\bigg\}.
\end{align*}

Given \eqref{sec5:o1} and \eqref{sec5:o2}, to show that the restriction of the infinitesimal generator of $\pi_t$ to the adelic Schwartz-Bruhat functions acts as $-H_\A$ on this space, it suffices to show that \begin{equation}\label{potential-limit-estimate}\frac{1}{t}\int_{D_\A[0,\infty)}\bigg(\int_0^tv(\omega(s))\,{\rm d}s\bigg) \alpha(\omega(t))\,{\rm d}P^\A_x \to v(x)\alpha(x)\end{equation} as $t$ tends to 0 from the right.

Since both $\alpha$ and $v$ are Schwartz-Bruhat functions on $\A$, there are natural numbers $N$ and $n$ so that on the set $U$ with \[U = \prod_{i = 1}^N{B_{p_i^{-n}}(x_i)}\times \prod_{i=N+1}^\infty \mathds Z_{p_i},\] the functions $\alpha$ and $v$ are constant and respectively equal to $\alpha(x)$ and $v(x)$.  Denote by $B_t$ the set of all $\omega$ in $D([0,\infty)\colon \A)$ that remain in $U$ for all $s$ in $[0,t]$.  Since $\sum_{i = 1}^\infty \sigma_i\alpha_i$ is finite, Proposition~\ref{three:ExitEqual} implies that %
\begin{align*}
\lim_{t\to 0^+} P(B_t) &= \lim_{t\to 0^+} \prod_{i=1}^N{\rm e}^{-\sigma_i\alpha_i tp_i^{nb}}\prod_{i = N+1}^\infty {\rm e}^{-\sigma_i\alpha_i t} \\&\ge  \lim_{t\to 0^+} \prod_{i=1}^N{\rm e}^{-\sigma_i\alpha_i tp_i^{nb}} {\rm e}^{- t\sum_{i = 1}^\infty \sigma_i\alpha_i} \to 1
\end{align*}
as $t$ tends to 0 from the right. Consequently, \begin{equation}\label{little-set-near-x}\lim_{t\to 0^+} P(B_t) = 1 \quad {\rm and}\quad \lim_{t\to 0^+} P(D_\A[0,\infty)\setminus B_t) = 0.\end{equation}

To calculate the limit \eqref{potential-limit-estimate}, write the integral on the left hand side of  \eqref{potential-limit-estimate} as a sum of integrals over a disjoint partition of $D_\A[0,\infty)$ given by
\begin{align}\label{DJ:Sum}
&\frac{1}{t}\int_{D_\A[0,\infty)} \bigg(\int_0^t v(\omega(s))\,{\rm d}s\bigg) \alpha(\omega(t))\,{\rm d}P^\A_x(\omega) \notag\\&\qquad= \int_{D_\A[0,\infty)} \Big(\frac{1}{t}\int_0^t v(\omega(s))\,{\rm d}s\Big)\alpha(\omega(t))\,{\rm d}P^\A_x(\omega)\notag\\
&\qquad = \int_{B_t} \bigg(\frac{1}{t}\int_0^t v(\omega(s))\,{\rm d}s\bigg)\alpha(\omega(t))\,{\rm d}P^\A_x(\omega)\notag\\ &\qquad\qquad\qquad\qquad\qquad\qquad\qquad+ \int_{D_\A[0,t]\setminus B_t} \bigg(\frac{1}{t}\int_0^t v(\omega(s))\,{\rm d}s\bigg)\alpha(\omega(t))\,{\rm d}P^\A_x(\omega).\end{align}%
As Schwartz-Bruhat functions, both $v$ and $\alpha$ are respectively bounded in absolute value by constants $K_1$ and $K_2$. Use these bounds to estimate the second term of \eqref{DJ:Sum} by
\begin{align*}
\left|\int_{D_\A[0,\infty)\setminus B_t} \bigg(\frac{1}{t}\int_0^t v(\omega(s))\,{\rm d}s\bigg)\alpha(\omega(t))\,{\rm d}P^\A_x(\omega)\right| & \leq \int_{D_\A[0,\infty)\setminus B_t} \bigg(\frac{1}{t}\int_0^t K_1\,{\rm d}s\bigg)K_2\,{\rm d}P^\A_x(\omega)\\& =  \int_{D_\A[0,\infty)\setminus B_t} K_1K_2\,{\rm d}P_x(\omega)\to 0
\end{align*}
as $t$ tends to 0 from the right.  For any path $\omega$ in $B_t$, $v(\omega(s))$ and $\alpha(\omega(s))$ are respectively equal to $v(x)$ and $\alpha(x)$ as long as $s$ is in $[0, t]$, implying that 
\begin{align*}\int_{B_t} \bigg(\frac{1}{t}\int_0^t v(\omega(s))\,{\rm d}s\bigg)\alpha(\omega(t))\,{\rm d}P^\A_x(\omega) & = \int_{B_t} \bigg(\frac{1}{t}\int_0^t v(x)\,{\rm d}s\bigg)\alpha(\omega(t))\,{\rm d}P^\A_x(\omega)\\& = \int_{B_t} v(x)\alpha(x)\,{\rm d}P^\A_x(\omega) \to v(x)\alpha(x)
\end{align*}
as $t$ tends to 0 from the right.  Taking the sum of limits of both terms of \eqref{DJ:Sum} proves \eqref{potential-limit-estimate}.

Since $(\pi_t)_{t> 0}$ is self adjoint, its infinitesimal generator is a self adjoint operator.  Its infinitesimal generator restricted to $SB(\A)$ is $-H_\A$, and $-H_\A$ is essentially self adjoint on $SB(\A)$.  Therefore, the infinitesimal generator of $(\pi_t)_{t> 0}$ is $-H_\A$.

\end{proof}

\begin{theorem}\label{5:FKFormula}
The semigroup $(\pi_t)_{t> 0}$ is equal to the semigroup $\big(\e^{-tH_\A}\big)_{t> 0}$.  
\end{theorem}

Theorem~\ref{5:FKFormula} is the Feynman-Kac formula for the semigroup $\big(\e^{-tH_\A}\big)_{t> 0}$.  The fact that this semigroup is a semigroup of integral operators whose kernels are integrals over adelic Brownian bridges follows as a corollary of this theorem.

\begin{proof}
Young's convolution inequality, together with the fact that $v$ is non-negative, implies that the semigroup $(\pi_t)_{t> 0}$ is a contraction semigroup. Since strongly continuous contraction semigroups with equal infinitesimal generators are equal, Lemma~\ref{5:FK:a}, Lemma~\ref{5:FK:b}, and Lemma~\ref{5:FK:c} together imply the theorem.
\end{proof}

A \emph{simple adelic potential function} is a function $v$ on $\A$ that is a sum \[v = \sum_i \tau_i\tilde{v}_i\] where for each $i$, $v_i$ is a Schwartz-Bruhat function on $\mathds Q_{p_i}$ and \[\tilde{v}_i(x) = v_i(x_i),\] the functions $v_i$ are uniformly bounded in $i$, and $(\tau_i)$ is summable. As a shorthand, denote henceforth by $v_i$ the function $\tau_iv_i$.  An adelic potential, $V$, is \emph{simple} if it is the potential associated to a simple adelic potential function.  To compress notation, denote below by $D_i[0, \infty)$ the path space $D([0,\infty): \mathds Q_{p_i})$.

\begin{lemma}\label{Lem:prodpullsout}
If $v$ is a simple adelic potential function and $x$ and $y$ are in $\A$, then
\[\int_{D_\A[0,\infty)}\prod_{i\in\mathds N}\e^{-\int_0^tv_i(\omega_i(s))\,{\rm d}s}\,{\rm d}P_{t, x, y}^\A(\omega) = \prod_{i\in\mathds N}\int_{D_i[0,\infty)}\e^{-\int_0^tv_i(\omega_i(s))}\,{\rm d}P^{i}_{t, x_i, y_i}(\omega_i).\] 
\end{lemma}

\begin{proof}
Since $v$ is bounded on $\A$, for all $\varepsilon$ greater than 0 there is a natural number $N_1$ and a function $e_1$ on $\mathds N$ such that $N$ is larger than $N_1$ implies that \[\big|e_1(N)\big| < \varepsilon\] and
\begin{align}\label{sec5:N1est}
&\int_{D_\A[0,\infty)}\prod_{i\in\mathds N}\e^{-\int_0^tv_i(\omega_i(s))\,{\rm d}s}\,{\rm d}P_{t, x, y}^\A(\omega) \notag\\&\hspace{1.5in}= \int_{D_\A[0,\infty)} \prod_{i\leq N} \e^{-\int_0^tv_i(\omega_i(s))\,{\rm d}s}\cdot \prod_{i> N} \e^{-\int_0^tv_i(\omega_i(s))\,{\rm d}s}\,{\rm d}P_{t, x, y}^\A(\omega)\notag\\
&\hspace{1.5in}= \int_{D_\A[0,\infty)} \prod_{i\leq N} \e^{-\int_0^tv_i(\omega_i(s))\,{\rm d}s}\cdot \e^{-\int_0^t\sum_{i> N}v_i(\omega_i(s))\,{\rm d}s}\,{\rm d}P_{t, x, y}^\A(\omega)\notag\\
&\hspace{1.5in}= \int_{D_\A[0,\infty)} \prod_{i\leq N} \e^{-\int_0^tv_i(\omega_i(s))\,{\rm d}s} (1+ e_1(N))\,{\rm d}P_{t, x, y}^\A(\omega)\notag\\
&\hspace{1.5in}= \int_{D[0,\infty)} \prod_{i\leq N} \e^{-\int_0^tv_i(\omega_i(s))\,{\rm d}s} (1+ e_1(N))\,{\rm d}P_{t, x, y}(\omega)\notag\\
&\hspace{1.5in}= (1+ e_1(N))\prod_{i\leq N}\int_{D_i[0,\infty)} \e^{-\int_0^tv_i(\omega_i(s))\,{\rm d}s}\,{\rm d}P_{t, x_i, y_i}^i(\omega_i),
\end{align} where Theorem~\ref{6:thm:7} implies the penultimate equality and Fubini's theorem implies the ultimate equality.  There is, furthermore, a natural number $N_2$ and a function $e_2$ on $\mathds N$ so that $N$ is larger than $N_2$ implies that \[\big|e_2(N)\big| < \varepsilon\] and %
\begin{equation}\label{sec5:N2est}\prod_{i>N}\int_{D_i[0,\infty)} \e^{-\int_0^tv_i(\omega_i(s))\,{\rm d}s}\,{\rm d}P_{t, x_i, y_i}^i(\omega_i) = (1 + e_2(N)).\end{equation} 

Equalities \eqref{sec5:N2est} and \eqref{sec5:N1est} together imply that if $N$ is larger than the maximum of $N_1$ and $N_2$, then
\begin{align*}
&\left|\int_{D_\A[0,\infty)}\prod_{i\in\mathds N}\e^{-\int_0^tv_i(\omega_i(s))\,{\rm d}s}\,{\rm d}P_{t, x, y}^\A(\omega) - \prod_{i\in\mathds N}\int_{D_i[0,t]} \e^{-\int_0^tv_i(\omega_i(s))\,{\rm d}s}\,{\rm d}P_{t, x_i, y_i}^i(\omega_i)\right|\\
&\qquad = \bigg|(1+ e_1(N))\prod_{i\leq N}\int_{D_i[0,\infty)} \e^{-\int_0^tv_i(\omega_i(s))\,{\rm d}s}\,{\rm d}P_{t, x_i, y_i}^i(\omega_i) \\&\hspace{2in}- \prod_{i\leq N}\int_{D_i[0,\infty)} \e^{-\int_0^tv_i(\omega_i(s))\,{\rm d}s}\,{\rm d}P_{t, x_i, y_i}^i(\omega_i) - e_2(N)\bigg|\\
&\qquad = \bigg|e_1(N)\prod_{i\leq N}\int_{D_i[0,\infty)} \e^{-\int_0^tv_i(\omega_i(s))\,{\rm d}s}\,{\rm d}P_{t, x_i, y_i}^i(\omega_i) - e_2(N)\bigg|\\
&\qquad \leq |e_1(N)| + |e_2(N)| < 2\varepsilon.
\end{align*}
The arbitrary choice of $\varepsilon$ implies that \[\int_{D_\A[0,\infty)}\prod_{i\in\mathds N}\e^{-\int_0^tv_i(\omega_i(s))\,{\rm d}s}\,{\rm d}P_{t, x, y}^\A(\omega) = \prod_{i\in\mathds N}\int_{D_i[0,\infty)}\e^{-\int_0^tv_i(\omega_i(s))}\,{\rm d}P^{i}_{t, x_i, y_i}(\omega_i).\] 
\end{proof}

Suppose that $v_i$ is a positive, bounded, continuous function on ${\mathds Q}_{p_i}$ and that the potential $V_i$ acts on functions $\alpha$ in $SB({\mathds Q}_{p_i})$ by \[(V_i\alpha)(x_i) = v(x_i)\alpha(x_i).\] Define by $H_i$  the Schr\"{o}dinger operator \[H_i= \Delta_i + V_i.\] The results of \cite{DVW} specialize to the case of scalar potentials on $\mathds Q_{p_i}$ and imply that for any positive $t$ and $x_i$ in $\mathds Q_{p_i}$, if \[\big(\pi^i_t\alpha)(x_i) = \int_{D_i[0,\infty)} e^{-\int_0^tv_i(\omega(s))\,{\rm d}s} \alpha(\omega(t))\,{\rm d}P^i_{x_i}(\omega_i),\] then \[\big(\pi^i_t\big)_{t> 0} = \big(\e^{-tH_i}\big)_{t> 0}.\] Furthermore, 
\begin{align*} \big(\pi^i_t\alpha\big)(x_i) &= \int_{\mathds Q_{p_i}}\bigg\{\int_{D_i[0,\infty)} e^{-\int_0^tv_i(\omega(s))\,{\rm d}s}\,{\rm d}P^i_{t, x_i,y_i}(\omega_i)\bigg\} \rho^i(t,x_i-y_i)\alpha(y_i)\,{\rm d}y_i\\ &= \int_{\mathds Q_{p_i}} k_t^i(x_i, y_i) \alpha(y_i)\,{\rm d}y_i,\end{align*}where \begin{equation}\label{sec5:iker}k_t^i(x_i, y_i) = \bigg\{\int_{D_i[0,\infty)} e^{-\int_0^tv_i(\omega(s))\,{\rm d}s}\,{\rm d}P^i_{t, x_i,y_i}(\omega_i)\bigg\} \rho^i(t,x_i-y_i).\end{equation}

\begin{theorem}
If $V$ is a simple adelic potential, then $(\pi_t)$ is a semigroup of integral operators whose kernel, $k_t$, is a product of the kernels $k_t^i$ of the semigroups $(\pi_t^i)$.
\end{theorem}

\begin{proof}
Denote for each positive $t$ respectively by $k_t^i$ on $\mathds Q_{p_i}\times \mathds Q_{p_i}$ and $k_t$ on $\A\times \A$ the functions given by \eqref{sec5:iker} and by \[k_t(x, y) = \bigg\{\int_{D_\A[0,\infty)}\e^{-\int_0^tv(\omega(s))\,{\rm d}s}\,{\rm d}P^\A_{t, x, y}(\omega)\bigg\}\rho(t, x-y).\]%
Since $v$ is a simple adelic potential function, there are Schwartz-Bruhat functions $(v_i)$ so that for each $x$ in $\A$, $v(x)$ is the sum \[v(x) = \sum_i v_i(x_i).\] For any Schwartz-Bruhat function $\alpha$, represent $v$ by the given sum to obtain the equalities 
\begin{align}\
\big(\pi_t^\A \alpha\big)(x) & = \int_{D_\A[0,\infty)}\bigg\{\prod_{i\in\mathds N}\e^{-\int_0^tv_i(\omega_i(s))\,{\rm d}s}\alpha(\omega(t))\bigg\}{\rm d}P_x^\A(\omega)\nonumber\\
& = \int_\A\bigg\{\int_{D_\A[0,\infty)}\prod_{i\in\mathds N}\e^{-\int_0^tv_i(\omega_i(s))\,{\rm d}s}\alpha(\omega(t))\,{\rm d}P_{t, x, y}^\A(\omega)\bigg\}\rho(t,x-y)\,{\rm d}y\nonumber\\
& = \int_\A\bigg\{\int_{D_\A[0,\infty)}\prod_{i\in\mathds N}\e^{-\int_0^tv_i(\omega_i(s))\,{\rm d}s}\,{\rm d}P_{t, x, y}^\A(\omega)\bigg\}\alpha(y)\rho(t,x-y)\,{\rm d}y\nonumber\\
& = \int_\A\prod_{i\in\mathds N}\bigg\{\int_{D_i[0,\infty)}\e^{-\int_0^tv_i(\omega_i(s))}\,{\rm d}P^{i}_{t, x_i, y_i}(\omega_i)\bigg\}\alpha(y)\rho(t,x-y)\,{\rm d}y\label{five:eqn:Prodaa}\\
& = \int_\A\prod_{i\in\mathds N}k_t^i(x_i, y_i)\alpha(y)\,{\rm d}y,\nonumber
\end{align} 
where Lemma~\ref{Lem:prodpullsout} implies \eqref{five:eqn:Prodaa}.

\end{proof}


\begin{thebibliography}{}


%
\bibitem{agk}  Albeverio, S., Gordon, E.\,I., Khrennikov,  A.\,Yu.: \textsl{Finite-dimensional approximations of operators in the Hilbert spaces of functions on locally compact abelian groups.} Acta Appl.\,Math.\,64, no.\,1, 33--73, (2000).

%
\bibitem{alb}  Albeverio, S., Karwowski, W.: \textsl{A random walk on $p$-adics - the generator and its spectrum.} Stochastic Processes and their Applications \textbf{53} l--22, (1994).

%
\bibitem{ave2} Avetisov, V.\,A., Bikulov, A.\,Kh.: \textsl{On the ultrametricity of the fluctuation dynamic mobility of protein molecules.} Proc. Steklov Inst. Math. 265 (1), 75-81, (2009).

%
\bibitem{ave3} Avetisov, V.\,A., Bikulov, A.\,Kh., Kozyrev, S.\,V.: \textsl{Application of $p$-adic analysis to models of breaking of replica symmetry.} J. Phys. A 32 (50), 8785-8791, (1999).

%
\bibitem{ave4} Avetisov, V.\,A., Bikulov, A.\,Kh., Kozyrev, S.\,V.: \textsl{Description of logarithmic relaxation by a model of a hierarchical random walk.} Dokl. Akad. Nauk 368 (2), 164-167, (1999).

%
\bibitem{ave} Avetisov, V.\,A., Bikulov, A.\,Kh., Osipov, V.\,Al.: \textsl{$p$-adic description of characteristic relaxation in complex systems.} Journal of Physics A, vol. 36, no. 15, pp. 4239-4246, (2003).

%
\bibitem{BDW} Bakken, E., Digernes, T., Weisbart, D.: \textsl{Brownian motion and finite approximations of quantum systems over local fields.} Rev. Math. Phys. 29(5), 1750016 (2017).

%
\bibitem{bw}  Bakken, E., Weisbart, D.:  \textsl{$p$-Adic Brownian motion as a limit of discrete time random walks.} Communications in Mathematical Physics, Volume 369, Issue 2, 371-402 (2019).

%
\bibitem{bil1} Billingsley, P.: \textsl{Convergence of Probability Measures, Second Edition}. John Wiley $\&$ Sons, (1999).

%
\bibitem{blair} Blair, A.\,D.: \textsl{Adelic path space integrals.} Rev. Math. Phys. 7, 21--49 (1995).

%
\bibitem{cam}  Cameron, R.\,H.: \textsl{A family of integrals serving to connect the Wiener and Feynman integrals.} J. Math. Phys. 39, 126--140 (1960).

%
\bibitem{Connes} Connes, A.: \textsl{Trace formula in noncommutative geometry and the zeros of the Riemann zeta function.} Selecta Math. (N.S.), 5, 29--106 (1999).

%
\bibitem{DHV} Digernes, T., Husstad, E., Varadarajan, V.\,S.: \textsl{Finite approximation of Weyl systems.} Mathematica Scandinavica. vol. 84 (2)  (1999).

%
\bibitem{DVV} Digernes, T., Varadarajan, V.\,S., Varadhan, S.\,R.\,S.: \textsl{Finite approximations to quantum systems}, Rev. Math. Phys. 6, no. 4, 621--648 (1994).

%
\bibitem{DVW} Digernes, T., Varadarajan, V.\,S., Weisbart, D.\,E.: \textsl{Schr\"{o}dinger operators on local fields: self adjointness and path integral representations for propagators}. Infin. Dimens. Anal. Quantum Probab. Relat. Top. 11, no. 4, 495 -- 512 (2008).

%
\bibitem{dra:survey} Dragovich, B., Khrennikov, A.\,Yu., Kozyrev, S.\,V., Volovich, I.\,V.: \textsl{On $p$-adic mathematical physics.} $p$-Adic Numbers, Ultrametric Analysis, and Applications. Volume 1, Issue 1, 1--17, (2009).%

%
\bibitem{DR} Dragovich, B., Raki\'{c}, Z.: \textsl{Path integrals for quadratic lagrangians on $p$-adic and adelic spaces.} $p$-Adic Numbers, Ultrametric Analysis, and Applications. 2 (4): 322--340 (2010).

%
\bibitem{DDN} Djordjevic, G.\,S., Dragovich, B., Nesic, Lj.: \textsl{Adelic path integrals for quadratic lagrangians.} Infin. Dim. Anal. Quant. Prob. Relat. Top. 6, 179--195 (2003).

%
\bibitem{feyn} Feynman, R.\,P., Hibbs, A.\,R.: \textsl{Quantum Mechanics and Path Integrals} (McGraw-Hill, New York, 1965). Emended by D.F. Styer (Dover, Mineola, New York, 2010).
 
%
\bibitem{Haran} Haran, S.: \textsl{Riesz potentials and explicit sums in arithmetic.} Invent. Math. 101, no. 3, 697--703 (1990).

%
\bibitem{lap} Johnson, G.\,W., Lapidus, M.\,L.: \textsl{The Feynman Integral and Feynman's Operational Calculus.} Oxford University Press, New York (2000). 

%
\bibitem {KKZ}   Khrennikov, A., Kozyrev, S., Z\'{u}\~{n}iga-Galindo, W.\,A.: \textsl{Ultrametric Equations and its Applications.} Encyclopedia of Mathematics and its Applications vol. 168, Cambridge University Press (2018).

%
\bibitem{koch92} Kochubei, A.\,N.: \textsl{Parabolic equations over the field of $p$-adic numbers.} Math. USSR Izvestiya 39, 1263--1280, (1992).

%
\bibitem{Meu} Meurice, Y.: \textsl{A path integral formulation of $p$-adic quantum mechanics.} Phys. Lett. B. 245, 99--104 (1990).

%
\bibitem{par} Parisi, G.: \textsl{On $p$-adic functional integrals.} Mod. Phys. Lett. A, pp. 639--643 (1988).

%
\bibitem{SC1} Saloff-Coste, L.: \textsl{Op\'{e}rateurs pseudo-diff\'{e}rentiels sur un corps local.} C. R. Acad. Sci. Paris S\'{e}r. I 297 (1983), 171-174.

%
\bibitem{SC2} Saloff-Coste, L.: \textsl{Op\'{e}rateurs pseudo-diff\'{e}rentiels sur certains groupes totalement discontinus.} Studia Math. 83, (1986), 205-228.

%
\bibitem{Sch1} Schwinger, J.: Proc. Nat. Acad. Sci. USA 45 (1959), 1552; 46 (1960), 261, 570, 893, 1401; 47 (1961), 1075; 48 (1962), 603.

%
\bibitem{Sch2} Schwinger, J.: \textsl{Quantum kinematics and dynamics.} W. A. Benjamin, New York, (1970).

%
\bibitem{Sch3} Schwinger, J.: \textsl{Unitary operator bases.} Proc. Natl. Acad. Sci. U.S.A. 46, 570 (1960).

%
\bibitem{Taib} Taibleson,  M.H.: \textsl{Fourier Analysis on Local Fields.} Princeton University Press, Princeton, N.J.;
University of Tokyo Press, Tokyo, (1975).

%
\bibitem{Zun:a} Torba, S.\,M., Z\'{u}\~{n}iga-Galindo, W.\,A.:  \textsl{Parabolic type equations and Markov stochastic processes on adeles}, Journal of Fourier Analysis and Applications 19 (4), 792--835 (2013).%

%
\bibitem{VSV-SW} Varadarajan, V.\,S.: \textsl{Variations on a theme of Schwinger and Weyl.} Lett. Math. Phys. {\bf 34}, 319--326 (1995). 

%
\bibitem{var97} Varadarajan, V.\,S.: \textsl{Path integrals for a class of $p$-adic Schr{\"o}dinger equations.} Lett. Math. Phys. 39, no. 2, 97--106, (1997).

%
\bibitem{Vlad88} Vladimirov, V.\,S.: \textsl{Generalized functions over the field of $p$-adic numbers.} Russian Math. Surveys, 43:5, 19--64, (1988).

%
\bibitem{Vlad90} Vladimirov, V.\,S.: \textsl{On the spectrum of some pseudo-differential operators over $p$-adic number field.} Algebra and analysis 2, 107--124, (1990).

%
\bibitem{VV89a} Vladimirov, V.\,S., Volovich, I.\,V.: \textsl{$p$-Adic quantum mechanics.} Comm. Math. Phys., 123:4, 659--676, (1989).

%
\bibitem{VV89b} Vladimirov, V.\,S., Volovich, I.\,V.: \textsl{$p$-Adic Schr\"{o}dinger-type equation.} Lett. Math. Phys., 18:1, 43--53, (1989).

%
\bibitem{vol} Volovich, I.\,V.: \textsl{Number theory as the ultimate physical theory.} CERN-TH.4781/87, Geneva, (July 1987).

%
\bibitem{DWexit} Weisbart, D.: \textsl{Estimates of certain exit probabilities for $p$-adic Brownian bridges.} (Under review), arXiv:2004.02265v2.

%
\bibitem{Weyl} Weyl, H.: \textsl{Theory of Groups and Quantum Mechanics.} Dover, New York, (1931).

%
\bibitem{zel} Zelenov, E.\,I.: \textsl{$p$-Adic path integrals.} J. Math. Phys., 32, pp. 147--152, (1991).




\end{thebibliography}
\end{document}